\documentclass[10pt]{article}
\DeclareMathSizes{10}{9}{7}{5}
\usepackage{amsmath,amssymb,amsthm,amsfonts}
\usepackage{bm}
\usepackage{mathrsfs}
\usepackage{graphicx}
\usepackage{mathtools}
\usepackage{braket}
\usepackage{comment}
\usepackage{enumitem}
\setlist[enumerate,1]{label=(\roman*),font=\upshape}
\usepackage{indentfirst}
\usepackage{booktabs}
\usepackage{placeins}
\usepackage{xurl}
\usepackage{authblk}
\usepackage[hidelinks]{hyperref}
\usepackage{bookmark}
\usepackage{cleveref}

\theoremstyle{plain}
\newtheorem{thm}{Theorem}[section]
\newtheorem{pro}{Proposition}[section]
\newtheorem{lem}{Lemma}[section]
\newtheorem{cor}{Corollary}[section]

\theoremstyle{definition}

\newtheorem{hyp}{Assumption}

\newcommand{\hypref}[1]{\textup{\ref{#1}}}
\newtheorem{rem}{Remark}[section]
\numberwithin{equation}{section}

\title{A Quasi Maximum Likelihood Estimation Method for Bergomi-Type Volatility Models}

\author[1]{Masaaki Fukasawa\thanks{Supported by KAKENHI, Grant Number 21K03369.}}
\author[1]{Haruki Tomita\thanks{Supported by JST SPRING, Grant Number JPMJSP2138.}}

\affil[1]{Graduate School of Engineering Science, The University of Osaka}

\date{}

\begin{document}
\maketitle
\section{Introduction}
Modeling the dynamics of volatility is a central issue in mathematical finance, both for derivative pricing and for the statistical analysis of option markets. 
A particularly important class of models is given by forward variance models, in which the dynamics of the variance curve are specified directly. 
This framework, often referred to as the Bergomi-type volatility framework, is widely used in derivatives modeling because it provides a flexible way to describe the joint behavior of spot volatility, forward variance, and implied volatility surfaces; see, for example, \cite{Bergomi2016}.
In this paper, we develop a model-based statistical methodology for estimating parameters in Bergomi-type volatility models from time-series observations of option prices.

Most existing statistical methods for volatility models are based on time series of the underlying asset or on volatility proxies constructed from high-frequency price data. 
Statistical inference based directly on option-price observations is much less developed. 
In this direction, recent work has developed a model-free method that exploits cross-sectional information from short-dated option portfolios~\cite{10.3150/24-BEJ1762}.
Our approach is different: we use time-series observations of option prices and develop a model-based procedure for estimating the parameters of a Bergomi-type volatility model. 
Thus, our focus is not on cross-sectional option-based inference, but on the high-frequency time-series estimation problem generated by the Bergomi-type dynamics.

To formulate the problem, we consider the following Bergomi-type stochastic volatility model:
\begin{equation}
\label{OWEIrngodfigmoAIBNoisdfbgoiAOINSdf}
\frac{dS_t}{S_t}=\sqrt{V_t}dZ_t,\
\frac{dV_t^u}{V_t^u}=k(\theta,u-t)dW_t,\ t\le u,\
V_t=\lim_{u\downarrow t}V_t^u,
\end{equation}
where \((Z,W)\) is a correlated two-dimensional Brownian motion, \(S_t\) is the discounted asset price, \(V_t^u\) denotes the forward variance with maturity \(u\), and \(V_t\) is the spot variance. 
The kernel \(k\) depends on an unknown parameter \(\theta\), which is the main object of statistical inference. 
It determines how shocks to the Brownian driver propagate across the forward variance curve. 
The flexibility in choosing the parametric form of \(k\) allows Bergomi-type models to capture a wide range of leverage structures, which is one of the reasons why they are well known in financial practice.

A key role in our analysis is played by the cumulative forward variance $I_t^T \coloneqq \int_t^T V_t^u\,du$.
It is known that this quantity admits the option-price representation
\[
I_t^{T}
=
2\int_0^{S_t}
E\!\left[(K-S_T)_+ \mid \mathscr{F}_t\right]\frac{dK}{K^2}
+
2\int_{S_t}^{\infty}
E\!\left[(S_T-K)_+ \mid \mathscr{F}_t\right]\frac{dK}{K^2}.
\]
Hence, at least conceptually, time-series observations of option prices allow us to recover discrete observations 
\(\{I_{t_i}^{T_j}\}\) of the cumulative forward variance. 
In actual market data, however, this reconstruction is not straightforward, due for example to market microstructure noise, the discreteness of available strikes, and the finite range of traded strikes. 
Therefore, in empirical applications we construct discrete observations of cumulative forward variance from option prices using the method of~\cite{fukasawa-murayama}.

Our first contribution is to show that the process \(I_t^T\) solves an infinite-dimensional stochastic differential equation:
\[
I_t^T
=
I_0^T-\int_0^t V_s\,ds+\int_0^t \sigma_s^T(\theta)\,dW_s,\qquad t<T,
\]
where for $t <x$, $\sigma_t^x(\theta)
\coloneqq
\int_t^x k(\theta,u-t)(dI_t^u/du)\,du$.
Since \(I_t^T\) can be reconstructed, at least approximately, from observed option prices, a finite collection of observations 
\(\{I_{t_i}^{T_j}\}\) can be regarded as market-based observations of this infinite-dimensional state process. 
This provides a route to statistical inference on the kernel parameter \(\theta\). 
At the same time, the resulting inference problem is highly nonstandard. 
The state equation is infinite-dimensional, whereas the driving Brownian motion is one-dimensional. 
Thus, the induced conditional covariance structure is highly degenerate. 
To the best of our knowledge, existing statistical theory does not directly address parameter estimation for such infinite-dimensional SDEs with rank-degenerate noise. 
Yet this type of equation arises naturally when one describes the time evolution of cumulative forward variance under a Bergomi-type model.

The statistical problem considered here is related in spirit to the parametric inference theory for discretely observed diffusion processes. 
For finite-dimensional diffusion processes, estimation from discrete and high-frequency observations has been extensively studied; see, for example, \cite{Yoshida1992,genon1993estimation,Kessler1997}. 
However, the present setting differs substantially from the classical finite-dimensional case. 
The observations are indexed both by time and maturity, and the effective noise dimension is much smaller than the dimension of the observed maturity vector. 
This degeneracy is a defining feature of the problem and prevents a direct application of classical inference methods.

To separate the statistical difficulty caused by this degenerate observation structure from additional analytical complications caused by singular kernels, such as the power-law kernel appearing in the rough Bergomi model~\cite{Bayer02062016}, we first develop the theory for a regular class of kernels. 
More precisely, we assume
$k\in C^{2,2}(\Theta\times[0,1])$,
$k>0$,
and consider a general parameter \(\theta\in\Theta\). 
This regular framework allows us to establish the core inferential theory.

Following the discretization strategy of~\cite{genon1993estimation}, we discretize the above infinite-dimensional SDE simultaneously in time and maturity over grids
$t_i=i\Delta$, 
$\Delta=1/n$,
$i=0,\ldots,n$,
$0=T_0<\cdots<T_d\le1$.
This yields an approximation in which, for each \(i\), the increment vector
$\Delta I_i^d
\coloneqq
[(I_{t_i}^{T_j}-I_{t_{i-1}}^{T_j})/{\sqrt{\Delta}}]_{j}$
satisfies
$
(I_{t_i}^{T_j}-I_{t_{i-1}}^{T_j})/\sqrt{\Delta}
\approx
\sigma_{t_{i-1}}^{T_j}(\theta)\Delta W_i / \sqrt{\Delta}.$
Thus, conditionally on the past, \(\Delta I_i^d\) behaves like a Gaussian vector with covariance matrix of the form \(\sigma\sigma^\mathsf{T}\) for some vector \(\sigma\in\mathbb{R}^{d\times1}\). 
Since this covariance matrix has rank one, the conditional Gaussian distribution is supported only on the one-dimensional space spanned by \(\sigma\), while the remaining \(d-1\) directions are degenerate. 
Consequently, a standard maximum likelihood approach cannot be applied directly.

To overcome this degeneracy, we introduce a nondegenerate proxy distribution obtained by redistributing the variance carried by the direction \(\mathrm{span}(\sigma)\) equally over the \(d\) maturity coordinates. 
We then define the estimating function as the log-likelihood associated with this proxy distribution:
\[
\begin{split}
U_{n,d}^{\epsilon}(\xi)
&\coloneqq
\frac{1}{n}\sum_{i=1}^n
\left[
\log\left(\frac{|\hat{\sigma}_{t_{i-1}}(\xi)|^2}{d}+\epsilon\right)
+
\frac{|\Delta I_i^d|^2+\epsilon d}
{|\hat{\sigma}_{t_{i-1}}(\xi)|^2+\epsilon d}
\right],
\qquad \epsilon>0,\\
\hat{\sigma}_t^j(\xi)
&\coloneqq
\sum_{l=1}^j
k(\xi,T_l-t)
\left(I_t^{T_l}-I_t^{T_{l-1}}\right)
\approx
\int_t^{T_j}
k(\xi,u-t)\frac{dI_t^u}{du}\,du
=
\sigma_t^{T_j}(\xi).
\end{split}
\]
The estimator is defined as a minimizer of this estimating function, and the main results of this paper establish its consistency and asymptotic mixed normality.

The final part of the paper examines the finite-sample behavior of the proposed method and applies it to SPXW option data. In the empirical analysis, we consider a shifted power-law kernel of the form
$k(t)=\eta(t+c)^{H-1/2},\ \eta>0,\ H\in\mathbb{R}$,
and estimate the kernel parameters $(\eta,H)$.
This specification is motivated by recent work on path-dependent volatility modeling, where shifted power-law kernels have been used to capture the memory structure of volatility; see, for example, \cite{AbiJaberLi2025,Guyon02092023,GuyonElAmrani2023}.

Section~\ref{DOfgindoghimOSINDoidfgh} states the main results. 
Section~\ref{OetidftghpmOAINsdofig} develops the key inequalities required for the proofs. 
Sections~\ref{02} and~\ref{OeigoainOISFDMgodfighoISB} establish consistency and asymptotic mixed normality, respectively. 
Finally, Section~\ref{eoitrgnoIAnoewirgnoISANOSIDFB} reports simulation results and empirical estimates based on market option data.
\section{Statement of Results}
\label{DOfgindoghimOSINDoidfgh}
\subsection{Notation and Assumptions.} We first describe the general setting.\ Fix $q \geq 1$ arbitrarily. The parameter space is a compact convex subset $\Theta$ of $\mathbb{R}^q$.\ Let $(\Omega,\mathscr{F},(\mathscr{F}_t),P)$ be a filtered probability space, and $W$ be a one-dimensional $(\mathscr{F}_t)$-standard Brownian motion.\ For simplicity, we assume that the kernel $k$ satisfies the following.  
\begin{hyp}
\label{EIrugnoIASNOsirgmoiNAOISmfgoiN}
    $k \in C^{2,2}(\Theta \times [0,1])$,\ $k >0$.
\end{hyp}
Furthermore,\ we impose the following identifiability condition,\ which will be required to the proof of consistency.
\begin{hyp}
\label{ERogienthoimOAINEOITRgdnofgnoINWSR}
    $\xi \neq \theta \implies \exists \delta >0$ such that $k(\xi,\cdot) - k(\theta,\cdot)$ has a constant nonzero sign on $(0,\delta)$.
\end{hyp}
Hereafter,\ for any $\alpha=1,\ldots,q$,\ the partial derivative with respect to $\xi_\alpha$ will be denoted by $\partial_\alpha$.\ In addition,\ the following assumption is necessary to prove the asymptotic mixed normality to be discussed later.
\begin{hyp}
\label{ErogisdotihnoiSMDOGFIsnrofgidsnfoginOIW}
    $\forall \theta \in \Theta,\ \forall x \in \mathbb{R}^q,\ x \neq 0,\ \exists \delta >0$ such that $\sum_{\alpha} x_\alpha \partial_\alpha k(\theta,\cdot)$ has a constant nonzero sign on $(0,\delta)$. 
\end{hyp}
\begin{rem}
The shifted power-law kernel
$k(t)=\eta(t+c)^{H-1/2},\ \eta>0,\ H\in\mathbb{R}$,
with a fixed shift parameter \(c>0\) also satisfies \ref{ErogisdotihnoiSMDOGFIsnrofgidsnfoginOIW}. Indeed, we have
$\partial_\eta k(t)=(t+c)^{H-1/2},\
\partial_H k(t)=\partial_\eta k(t)\eta\log(t+c),
$
and for any \(x=(x_1,x_2)\neq0\),
$
x_1\partial_\eta k(t)+x_2\partial_H k(t)
=
(t+c)^{H-1/2}\{x_1+\eta x_2\log(t+c)\}
$.
Since the first factor is strictly positive, the sign is determined by
$g(t) \coloneqq x_1+\eta x_2\log(t+c)$.
If \(g(0)\neq0\), then continuity implies that \(g\) has a constant nonzero sign on \((0,\delta)\) for some sufficiently small \(\delta>0\). 
If \(g(0)=0\), then \(x_2\neq0\) and
$g(t)=\eta x_2\log\left(1+t / c\right)$,
which has the constant nonzero sign of \(x_2\) for all \(t>0\). 
Thus the required sign condition holds.
\end{rem}
Let us fix an arbitrary true value $\theta \in \Theta$,\ and define
\begin{equation}
    \label{IWurgnsitugbiUQNIUBAiunsifgubW}
    V_t^u \coloneqq V_0^u \exp \left(\int_0^t k(\theta,u-s) dW_s -\frac{1}{2} \int_0^t k(\theta,u-s)^2ds\right),\ 0 \leq t \leq u \leq 1.
\end{equation}
Here, suppose that $V_0^\cdot$ is positive,\ non-random and continuous on $[0,1]$.\ Furthermore, we define 
\begin{equation}
    \label{OEIRgndosfihgmfdoghinoIWMOSIngodisghoiNSOIN}
    I_t^T \coloneqq \int_t^T V_t^u du,\ 0 \leq t <T \leq 1,
\end{equation}
and extend the definition by setting $I_t^T =0 ,\ t \geq T$.\ Then the following holds.
\begin{pro}
\label{SOfdignrsothimoISNogisdnfgoimOSINFgoidsfg}
   Under \hypref{EIrugnoIASNOsirgmoiNAOISmfgoiN},\ for each $t < T$, the following holds almost surely.
   \[
    I_t^T =I_0^T -\int_0^t V_s ds + \int_0^t \left(\int_s^T k(\theta,u-s)V_s^u du\right)dW_s,
\]
where $V_t \coloneqq V_t^t$.
\end{pro}
\noindent The proof is given in Section~\ref{OetidftghpmOAINsdofig}.

We then define 
\begin{equation}
    \label{OERignrotihmoIAMOWINrogidmfgoiNABOIB}
    \sigma_t^x(\xi) \coloneqq \int_t^x k(\xi,u-t) V_t^u du,\ 0 \leq t<x \leq 1,\ \xi \in \Theta,
\end{equation}
here we set $\sigma_t^x(\xi) =0,\ t \geq x$.\ We now define the partition of the time  by $\Delta \coloneqq 1/n,\ t_i \coloneqq i\Delta$ and suppose the partition of the dimension as follows.
\begin{hyp}
\label{EORigndgfhoimOAINWOirngdofgimOASINSOfiogm}
Assume that $0=T_0 <T_1 <\cdots <T_d \leq 1$, and there exists a measure $\mu$ having positive density,\ the probability measures $\mu_d \coloneqq \frac{1}{d}\sum_{j=1}^d \delta_{T_j}$ on $[0,1]$ weakly converge to $\mu$.
\end{hyp}
For the proof of asymptotic mixed normality,\ we further need the subsequent assumption.
\begin{hyp}
\label{OEIrgnrsotihymoIANOisrfg}
    $\max_k |T_k-T_{k-1}| =o(1/\sqrt{n})$.\ In particular,\ for the uniform partition,\ $\sqrt{n} \ll d$.
\end{hyp}
Accordingly, for $j=1,\ldots,d$, we put the discretization of \eqref{OERignrotihmoIAMOWINrogidmfgoiNABOIB} as 
\begin{equation} 
\label{EOritgnsrotiyhmoiNSOISngodsifhg}
        \hat{\sigma}_t^{j}(\xi) \coloneqq \sum_{l=1}^j k(\xi,T_l-t) (I_t^{T_l}-I_t^{T_{l-1}}).
\end{equation}
\subsection{Convergence of Contrast Functions.}For $\Delta I_i^d \coloneqq [(I_{t_i}^{T_j}-I_{t_{i-1}}^{T_j})/\sqrt{\Delta}]_j$ and $\epsilon >0$, we define the estimating function $U_{n,d}^\epsilon$ as 
\begin{equation}
\label{WRgudgtnhnOAIdnfozdifgNSODIndfhg}
    \begin{split}
U_{n,d}^\epsilon(\xi) &\coloneqq \frac{1}{n} \sum_{i=1}^n F_\epsilon(\xi,t_{i-1},\Delta I_i^d),\\
F_\epsilon(\xi,t,x) &\coloneqq \log\left(\frac{|\hat{\sigma}_t(\xi)|^2}{d}+\epsilon\right) + \frac{|x|^2+\epsilon d}{|\hat{\sigma}_t(\xi)|^2+\epsilon d} .
\end{split}
\end{equation}
Furthermore,\ for $\sigma_t^d(\xi)\coloneqq [\sigma_t^{T_j}(\xi)]_j$, we set 
\[
    U_{t,d}^\epsilon(\xi)\coloneqq F_\epsilon(\xi,t,\sigma_t^d(\theta)),\ U_d^\epsilon(\xi) \coloneqq \int_0^1 U_{t,d}^\epsilon(\xi)dt,
\]
and define 
   \begin{equation}
    \label{WOIrgneotigmOIQOEIrgoiMASOIDSg}
    U^\epsilon(\xi) \coloneqq \int_0^1 \log\left(\int_0^1 \sigma_t^x(\xi)^2 \mu(dx)+\epsilon\right)+\frac{\int_0^1 \sigma_t^x(\theta)^2\mu(dx)+\epsilon}{\int_0^1 \sigma_t^x (\xi)^2 \mu(dx)+\epsilon} dt.
\end{equation}
Then the following statement holds.
\begin{thm}
\label{EORigeothimoISNAOFISNrogsdifghaosfigneso}
    Under~\hypref{EIrugnoIASNOsirgmoiNAOISmfgoiN} and \hypref{EORigndgfhoimOAINWOirngdofgimOASINSOfiogm},\ for any $\epsilon >0$,\ $U_{n,d}^\epsilon$ converges uniformly in probability to $U^\epsilon$ on $\Theta$ as $n,d \to \infty$.
\end{thm}
\noindent The proof is given in Section~\ref{02}.

Here,\ since $\Theta$ is compact,\ a minimizer of $U_{n,d}^\epsilon(\cdot)$ exists,\ which we denote by $\hat{\theta}_{n,d}^\epsilon$.\ Then we have:
\begin{cor}
\label{EORginseoihmOIANOisrgmoiNAOiedrg}
        Under~\hypref{EIrugnoIASNOsirgmoiNAOISmfgoiN},\ \hypref{ERogienthoimOAINEOITRgdnofgnoINWSR} and \hypref{EORigndgfhoimOAINWOirngdofgimOASINSOfiogm},\ for all $\epsilon >0$,\ $\hat{\theta}_{n,d}^\epsilon$ is weakly consistent for the true value $\theta$.
\end{cor}
\begin{rem}
    In the above,\ we have implicitly fixed the equivalent martingale measure as the probability measure $P$.\ However,\ since convergence in probability is preserved under any equivalent measure,\ the above statement also holds under the original probability measure.
\end{rem}
\subsection{Asymptotic Mixed Normality.}For $\epsilon  \geq 0$ and $\xi \in \Theta$,\ we define non-negative definite $q \times q$ random matrices $B^\epsilon(\xi)$ and $D^\epsilon(\xi)$:
\begin{equation}
\label{ErogirethoimOIWNOEIrgndofighoIASN}
   \begin{split}
    &B^\epsilon(\xi) \coloneqq  \int_0^1y_t^\epsilon(\xi) E_t(\xi) dt,\ D^\epsilon(\xi) \coloneqq \int_0^1  z_t^\epsilon(\xi) E_t(\xi) dt,\\
    &E_t(\xi) \coloneqq e_t(\xi) e_t(\xi)^\mathsf{T},\ e_t(\xi) \coloneqq \left[\int_0^1 \partial_\alpha \sigma_t^x(\xi)\sigma_t^x(\xi)\mu(dx)\right]_{\alpha},\\
    &y_t^\epsilon(\xi) \coloneqq \frac{4}{\left(\int_0^1 \sigma_t^x(\xi)^2\mu(dx)+\epsilon\right)^2},\ z_t^\epsilon(\xi) \coloneqq \frac{8\left(\int_0^1 \sigma_t^x(\xi)^2\mu(dx)\right)^2}{\left(\int_0^1 \sigma_t^x(\xi)^2\mu(dx)+\epsilon\right)^4}.
\end{split} 
\end{equation}
For simplicity,\ we write $B^\epsilon$ and $D^\epsilon$ for $B^\epsilon(\theta)$ and $D^\epsilon(\theta)$,\ respectively.
As proved in Section~\ref{OeigoainOISFDMgodfighoISB},\ under \hypref{EIrugnoIASNOsirgmoiNAOISmfgoiN},\ \hypref{ErogisdotihnoiSMDOGFIsnrofgidsnfoginOIW} and \hypref{EORigndgfhoimOAINWOirngdofgimOASINSOfiogm},\ for any $\epsilon \geq 0$,\ $B^\epsilon$ and $D^\epsilon$ are $P$-a.s. invertible.
\begin{thm}
\label{OWIRgnosdifgnoIAMOSIdf}
Assume~\hypref{EIrugnoIASNOsirgmoiNAOISmfgoiN}--\hypref{OEIrgnrsotihymoIANOisrfg}.\ Let $\theta$ be an interior point of $\Theta$ and let $\mathscr{F}^I \coloneqq \sigma(I_t^T,\ 0 \leq t <T \leq 1)$.\ Then for all $\epsilon >0$,\ $\sqrt{n}(\hat{\theta}_{n,d}^\epsilon-\theta)$ $\mathscr{F}^I$-stably converges in law to  $S^\epsilon$, where $S^\epsilon$ is defined on an extension of the space $(\Omega,\mathscr{F}^I,P)$ and is,
    conditionally on $\mathscr{F}^I$,\ centered Gaussian with covariance matrix
    \[
        \Gamma^\epsilon \coloneqq (B^\epsilon)^{-1} D^\epsilon  (B^\epsilon)^{-1}.
    \]
\end{thm}
\noindent The proof is given in Section~\ref{OeigoainOISFDMgodfighoISB}.

\begin{rem}
    For $\epsilon \geq 0$, we show that $\Gamma^\epsilon \geq \Gamma^0$. Fix $x \in \mathbb{R}^q$, and define $a \coloneqq (B^\epsilon)^{-1} x$ and $b \coloneqq 2(B^0)^{-1}x$. Writing $s_t = \int_0^1 \sigma_t^x(\theta)^2 \mu(dx)$, we have 
    \[x^\mathsf{T} \Gamma^\epsilon x = 8 \int_0^1 \frac{(e_t^\mathsf{T} a)^2 s_t^2}{(s_t+\epsilon)^4} dt,\quad x^\mathsf{T} \Gamma^0 x = 2\int_0^1 \frac{(b^\mathsf{T} e_t)^2}{s_t^2}dt,\]
    while, since $D^0 = 2B^0$, we have
    \[x^\mathsf{T} \Gamma^0 x = b^\mathsf{T} B^\epsilon a = 4 \int_0^1 \frac{b^\mathsf{T} e_te_t^\mathsf{T} a s_t}{s_t(s_t+\epsilon)^2}dt \leq \left( x^\mathsf{T} \Gamma^0 x\right)^{1/2} \left( x^\mathsf{T} \Gamma^\epsilon x\right)^{1/2}.\]
\end{rem}

For $\epsilon  > 0$ and $\xi \in \Theta$,\ we define non-negative definite $q \times q$ random matrices $\hat{B}_{n,d}^\epsilon(\xi)$ and $\hat{D}_{n,d}^\epsilon(\xi)$:
\[
    \begin{split}
    &\hat{B}_{n,d}^\epsilon(\xi) \coloneqq \frac{1}{n} \sum_{i=1}^n \hat{y}_{t_{i-1},d}^\epsilon (\xi) \hat{e}_{t_{i-1},d}(\xi)\hat{e}_{t_{i-1},d}(\xi)^\mathsf{T},\\
    &\hat{D}_{n,d}^\epsilon(\xi) \coloneqq \frac{1}{n} \sum_{i=1}^n \hat{z}_{t_{i-1},d}^\epsilon (\xi) \hat{e}_{t_{i-1},d}(\xi)\hat{e}_{t_{i-1},d}(\xi)^\mathsf{T},\\
    &\hat{e}_{t,d}(\xi) \coloneqq [\partial_\alpha \hat{\sigma}_t(\xi)^\mathsf{T} \hat{\sigma}_t(\xi)/d]_\alpha,\\
    &y_{t,d}^\epsilon(\xi) \coloneqq \frac{4}{(|\hat{\sigma}_t(\xi)|^2 /d + \epsilon)^2},\ z_{t,d}^\epsilon(\xi) \coloneqq \frac{8(|\hat{\sigma}_t(\xi)|^2 /d)^2}{(|\hat{\sigma}_t(\xi)|^2 /d + \epsilon)^4}.
    \end{split}
\]
Then,\ the following holds.
\begin{cor}
\label{EEGRTOAINoinoidsnSOAI}
    Under~\hypref{EIrugnoIASNOsirgmoiNAOISmfgoiN},\ \hypref{ERogienthoimOAINEOITRgdnofgnoINWSR} and \hypref{EORigndgfhoimOAINWOirngdofgimOASINSOfiogm},\ for all $\epsilon >0$,\ $\hat{B}_{n,d}^\epsilon(\hat{\theta}_{n,d}^\epsilon)$ and $\hat{D}_{n,d}^\epsilon(\hat{\theta}_{n,d}^\epsilon)$ converge in probability to $B^\epsilon$ and $D^\epsilon$,\ respectively, as $n,d \rightarrow \infty$.
\end{cor}
Here, suppose that \hypref{EIrugnoIASNOsirgmoiNAOISmfgoiN}--\hypref{OEIrgnrsotihymoIANOisrfg}.\ For $\epsilon >0$, let 
\begin{equation}
\label{OEirgndsoihoiNOISNdofgisadnfgoi}
     \hat{\Gamma}_{n,d}^\epsilon\coloneqq \hat{B}_{n,d}^\epsilon(\hat{\theta}_{n,d}^\epsilon)^{-1}\hat{D}_{n,d}^\epsilon(\hat{\theta}_{n,d}^\epsilon)\hat{B}_{n,d}^\epsilon(\hat{\theta}_{n,d}^\epsilon)^{-1}.
\end{equation}
Since $B^\epsilon$ and $D^\epsilon$ are nonsingular,\ $\hat{\Gamma}_{n,d}^\epsilon$ is well defined and invertible for sufficiently large $n$ and $d$. Then, by the asymptotic mixed normality,\ the studentized statistic converges in law to the standard normal; That is,
\begin{equation}
\label{EWOrigndoghimOAISNFOSIFg}
    \sqrt{n}(\hat{\Gamma}_{n,d}^\epsilon)^{-1/2}(\hat{\theta}_{n,d}^\epsilon - \theta) \xrightarrow[n,d \rightarrow \infty]{\mathcal{L}} \mathcal{N}(0,I_q).
\end{equation}
\section{Some Estimates}
\label{OetidftghpmOAINsdofig}
From now on, without loss of generality,\ we extend the kernel by setting $k(\xi,t) =0$ for $t<0$,\ for any $\xi$ and define $V_t^u =0,\ t>u$.
In this section,\ we assume \hypref{EIrugnoIASNOsirgmoiNAOISmfgoiN} unless otherwise specified.
\begin{lem}
\label{OWImeorignOIAMOIMSdfkLQ}
    There exists a $p$-integrable random variable $M$ for any $p \geq 1$,\ such that the following holds.
    \begin{enumerate}
    \item For all $t,T,T_1,T_2$,\ $|I_t^T| \leq M|T-t|$ and $|I_t^{T_2}-I_t^{T_1}| \leq M|T_2-T_1|$.
    \item For any $\xi,t,d$ and $\alpha,\beta=1,\ldots,q$, $|\hat{\sigma}_t(\xi)| \leq M\sqrt{d}$,\ $|\partial_\alpha \hat{\sigma}_t(\xi)| \leq M\sqrt{d}$ and $|\partial_\alpha \partial_\beta \hat{\sigma}_t(\xi)| \leq M\sqrt{d}$.
    \item For any $\xi,t,x$ and $\alpha,\beta=1,\ldots,q$,\\ $|\sigma_t^x(\xi)| \leq M|x-t|$,\ $|\partial_\alpha \sigma_t^x(\xi)| \leq M|x-t|$ and $|\partial_\alpha \partial_\beta \sigma_t^x(\xi)| \leq M|x-t|$.
    \end{enumerate}
\end{lem}
\begin{proof}
    (i) Let  $\mathbb{D} \coloneqq \{(t,u) \mid 0 \leq t \leq u \leq 1\}$ and $X(t,u) \coloneqq \int_0^t k(\theta,u-s) dW_s, (t,u) \in \mathbb{D}$.\ Since the Kolmogorov-Totoki theorem~\cite{Kunita2019}, $X$ admits a continuous modification.\ In particular,\ $M \coloneqq \sup_{(t,u) \in \mathbb{D} } V_t^u <\infty$.\ Moreover, from Borell-TIS inequality~\cite{AdlerTaylor2007},\ $M$ is $p$-integrable for every $p \geq 1$.\ Then for all $t<T$,\ $|I_t^T|=\left|\int_t^T V_t^udu\right| \leq M|T-t|$. Similarly, for every $t <T_1,T_2$,\ $|I_t^{T_2}-I_t^{T_1}| =\left|\int_{T_1}^{T_2}V_t^u du\right| \leq M|T_2-T_1|$.\ On the other hand, in the case of $T_1 \leq t <T_2$, it follows from the previous result that $|I_t^{T_2}-I_t^{T_1}| =|I_t^{T_2}| \leq M|T_2-t| \leq M|T_2-T_1|$.

    (ii) For arbitrary $j$,\ $\xi$ and $t$,\  since \hypref{EIrugnoIASNOsirgmoiNAOISmfgoiN} and the result of (i) implies there exists another constant $C>0$ such that
    \[|\hat{\sigma}_t^j(\xi)| \leq C \sum_{l=1}^j (I_t^{T_l}-I_t^{T_{l-1}})=CI_t^{T_j} \leq CM|T_j-t| \leq CM,\]
the assertion has been established. The rest can be proved similarly.

(iii) Again,\ by \hypref{EIrugnoIASNOsirgmoiNAOISmfgoiN},\ there exists the same constant $C>0$ as before in (ii) such that for any $\xi$ and $t<x$,
\[|\sigma_t^x(\xi)| \leq C\int_t^xV_t^udu=CI_t^x \leq CM|x-t|,\]
and hence the first claim follows.\ The remaining part can be proved analogusly.
\end{proof}

\begin{proof}[Proof of Proposition~\textnormal{\ref{SOfdignrsothimoISNogisdnfgoimOSINFgoidsfg}}]
    -- By Lemma~\ref{OWImeorignOIAMOIMSdfkLQ},\ the stochastic Fubini's theorem~\cite{Veraar01082012} can be applied and from \eqref{IWurgnsitugbiUQNIUBAiunsifgubW},\ for every $t < T$,
        \[\begin{split}
        \int_0^t \int_s^T k(\theta,u-s) V_s^u du dW_s &=\int_0^T \int_0^{t \wedge u}k(\theta,u-s)V_s^u dW_s du\\
        &=\int_0^t (V_u^u-V_0^u) du +\int_t^T (V_t^u -V_0^u)du\\
        &=\int_0^t V_u du +I_t^T -I_0^T.
    \end{split}\]
\end{proof}

\begin{lem}
\label{EOEITnhgoedtihmoIAOIDGHOIDFNhoiAMOIWS}
    The following holds.
    \begin{enumerate}
        \item For all $p \geq 1$,\ there exists a constant $C_p$ such that for any $t,s,T$,\ \[E[|I_t^T -I_s^T|^p] \leq C_p |t-s|^{p/2}.\]
        \item There exists a constant $C$ such that for every $\xi$ and $t,s,d$,\ \[E[|\hat{\sigma}_t(\xi)-\hat{\sigma}_s(\xi)|^2] \leq Cd|t-s|.\]
        \item There exists a constant $C$ such that for any $\xi,t,s,d$ and $\alpha=1,\ldots,q$,\ \[E[|\partial_\alpha\hat{\sigma}_t(\xi)-\partial_\alpha \hat{\sigma}_s(\xi)|^2] \leq Cd|t-s|.\]
        \item There exists a constant $C$ such that for all $\xi,t,s,d$ and $\alpha,\beta =1,\ldots,q$,\ \[E[|\partial_\alpha \partial_\beta \hat{\sigma}_t(\xi)-\partial_\alpha \partial_\beta \hat{\sigma}_s(\xi)|^2] \leq Cd|t-s|.\]
    \end{enumerate}
\end{lem}
\begin{proof}
Hereafter, $C$ denotes a generic positive constant and $M$ a generic random variable with finite moments of all orders; both may vary from line to line.

    (i) From Proposition~\ref{SOfdignrsothimoISNogisdnfgoimOSINFgoidsfg} and the Burkh\"{o}lder-Davis-Gundy inequality,\ for all $p \geq 1$,\ there exists a constant $K_p$ such that for any $s,t<T$,
    \[
        E[|I_t^T-I_s^T|^p] 
        \leq 2^{p-1} \left(E\left[\left|\int_s^t V_udu\right|^p\right]+K_pE\left[\left|\int_s^t \sigma_u^T(\theta)^2du\right|^{p/2}\right]\right).
 \]
    Therefore,\ by Lemma~\ref{OWImeorignOIAMOIMSdfkLQ} we obtain the desired inequality.\ If $s<T \leq t$, then by Lemma~\ref{OWImeorignOIAMOIMSdfkLQ} again, $|I_t^T-I_s^T|^p =|I_s^T|^p \leq M|T-s|^p \leq M|t-s|^p$,\ and thus the claim follows.

    (ii) For all $\xi$ and $t,s,j=1,\ldots,d$, set $\hat{\sigma}_t^j(\xi) -\hat{\sigma}_s^j(\xi) = A_1^j +A_2^j$, where
    \[\begin{split}A_1^j &\coloneqq \sum_{l=1}^j (k(\xi,T_l-t)-k(\xi,T_l-s))(I_t^{T_l}-I_t^{T_{l-1}}),\\ 
    A_2^j &\coloneqq \sum_{l=1}^j k(\xi,T_l-s)[(I_t^{T_l}-I_s^{T_l})-(I_t^{T_{l-1}}-I_s^{T_{l-1}})].\end{split}
\]
By the $\xi$-uniform Lipschitz continuity of $k$, the monotonicity of $I_t^T$ in $T$, and Lemma~\ref{OWImeorignOIAMOIMSdfkLQ}, we have $|A_1^j| \leq C|t-s|I_t^{T_j} \leq M|t-s|$. Hence, $E|A_1^j|^2 \leq C|t-s|$. Next, by summation by parts,
\[A_2^j = \sum_{l=1}^{j-1} (k(\xi,T_l-s)-k(\xi,T_{l+1}-s))(I_t^{T_l}-I_s^{T_l})+k(\xi,T_j-s)(I_t^{T_j}-I_s^{T_j}).\]
Using the $\xi$-uniform Lipschitz continuity and boundedness of $k$,
\[|A_2^j| \leq C\sum_{l=1}^{j-1} (T_{l+1}-T_l) |I_t^{T_l}-I_s^{T_l}| + C|I_t^{T_j} - I_s^{T_j}|.\]
Therefore, by the Cauchy-Schwarz inequality and $\sum_{l=1}^{j-1} (T_{l+1}-T_l) \leq 1$,
\[E|A_2^j|^2 \leq  C \sum_{l=1}^{j-1} (T_{l+1}-T_l)E|I_t^{T_l}-I_s^{T_l}|^2 + CE|I_t^{T_j}-I_s^{T_j}|^2 \leq C|t-s|.\]
Consequently, $E|\hat{\sigma}_t^j (\xi)- \hat{\sigma}_s^j(\xi)|^2 \leq C|t-s|$ and we obtain the desired estimate.
(iii) and (iv) are proved in the same way as (ii).
\end{proof}

\begin{lem}
\label{EOrignrotihomoISADOISnrgoin}
For any $t<x$,
   \[
   \begin{split}
       &\sigma_t^x =\sigma_0^x +\int_0^t b_s^xds +\int_0^t \phi_s^x dW_s,\\
       &b_s^x \coloneqq -k(\theta,0)V_s -\int_s^x \partial_tk(\theta,u-s)V_s^udu,\ \phi_s^x \coloneqq \int_s^x k(\theta,u-s)^2 V_s^udu.
   \end{split}\]
   Moreover,\ for all $p \geq 1$,\ there exists a constant $C_p$ such that for every $s,t,x$,\ $E[|\sigma_t^x-\sigma_s^x|^p] \leq C_p|t-s|^{p/2}$ and for any $s,t<x$,\ $E[|\phi_t^x-\phi_s^x|^p] \leq C_p|t-s|^{p/2}$.
\end{lem}
\begin{proof}
    In the following,\ $k(\theta,\cdot)$ will simply be denoted by $k(\cdot)$.\ First,\ for any $u \in (0,1]$ and each $t \leq u$,\ from \eqref{IWurgnsitugbiUQNIUBAiunsifgubW},\ integration by parts yields
   \begin{equation}
   \label{EORignethyoimOAIEnOSIRgnods}
       k(u-t)V_t^u=k(u)V_0^u -\int_0^tk'(u-s)V_s^uds +\int_0^t k(u-s)^2 V_s^udW_s.
   \end{equation}
   Therefore,\ by the stochastic Fubini's theorem (justified by Lemma~\ref{OWImeorignOIAMOIMSdfkLQ}),\ we obtain
   \begin{equation}\begin{split}
   \label{EOrtigeothimoIASNOSIFngoinertg}
              \sigma_t^x &=\int_t^x k(u-t)V_t^udu\\
       &=\int_t^xk(u)V_0^udu -\int_0^t\int_t^x k'(u-s)V_s^ududs+\int_0^t \int_t^x k(u-s)^2V_s^ududW_s
       \end{split}
   \end{equation}
   for all $t<x$.\ Here,\ in \eqref{EORignethyoimOAIEnOSIRgnods},\ if we set $t=u$,\ then by the stochastic Fubini's theorem again,\ we obtain
   \begin{equation}
   \label{ERogienthoimOASINSOfigndofgih}\begin{split}
\int_t^xk(u)V_0^udu &=\sigma_0^x -\int_0^t k(0)V_udu-\int_0^t \int_s^t k'(u-s)V_s^ududs\\
&\ \ \ \ +\int_0^t \int_s^tk(u-s)^2 V_s^u dudW_s.
\end{split}
   \end{equation}
   Hence,\ combining \eqref{EOrtigeothimoIASNOSIFngoinertg} and \eqref{ERogienthoimOASINSOfigndofgih},\ we arrive at
\[
   \begin{split}
       &\sigma_t^x =\sigma_0^x +\int_0^t b_s^xds +\int_0^t \phi_s^x dW_s,\\
       &b_s^x \coloneqq -k(0)V_s -\int_s^x k'(u-s)V_s^udu,\ \phi_s^x \coloneqq \int_s^x k(u-s)^2 V_s^udu.
   \end{split}\]
   Similarly,\ it follows that
   \[\begin{split}
       &\phi_t^x =\phi_0^x+\int_0^t c_s^x ds + \int_0^t \psi_s^x dW_s,\\
       &c_s^x \coloneqq -k(0)^2 V_s-2 \int_s^x k(u-s)k'(u-s) V_s^u du,\ \psi_s^x\coloneqq \int_s^x k(u-s)^3V_s^udu.
   \end{split}\]
   Therefore,\ as in the proof of Lemma~\ref{OWImeorignOIAMOIMSdfkLQ},\ it follows from the BDG inequality that there exists a constant $C_p$ such that for all $s,t<x$,\ $E[|\sigma_t^x-\sigma_s^x|^p] \leq C_p|t-s|^{p/2}$ and $E[|\phi_t^x-\phi_s^x|^p] \leq C_p|t-s|^{p/2}$.\ If $s<x \leq t$,\ by Lemma~\ref{OWImeorignOIAMOIMSdfkLQ},\ there exists a random variable $M$ such that $|\sigma_t^x-\sigma_s^x|^p=|\sigma_s^x|^p \leq M|x-s|^p \leq M|t-s|^p$ and thus the craim follows.
\end{proof}
\section{Convergence of Contrast Functions}
\label{02}
We shall prove the convergence of the estimating function~\eqref{WRgudgtnhnOAIdnfozdifgNSODIndfhg} in several steps. To this end,\ we set
\[
\chi_{i,d}^\epsilon (\xi) \coloneqq \frac{1}{n} [F_\epsilon(\xi,t_{i-1},\Delta I_i^d)-U_{t_{i-1},d}^\epsilon (\xi)]=\frac{1}{n} \cdot \frac{1}{|\hat{\sigma}_{t_{i-1}}(\xi)|^2 +\epsilon d}(|\Delta I_i^d |^2 -|\sigma_{t_{i-1}}^d(\theta)|^2),\\
\]
and put  $X_{n,d}^\epsilon (\xi) \coloneqq \sum_{i=1}^n \chi_{i,d}^\epsilon (\xi)$. 

Hereafter,\ we use the same symbol $C$ for all constants appearing in the estimates,\ and all random variables used in the estimates involving Lemma~\ref{OWImeorignOIAMOIMSdfkLQ} will be denoted by the same symbol $M$.

\begin{lem}
\label{EOirgrthoimOASINOSdigdghNOSINR}
    Under \hypref{EIrugnoIASNOsirgmoiNAOISmfgoiN},\ for each $\xi$ and $\epsilon$,
    \[\begin{split}
&\sum_{i=1}^n E[\chi_{i,d}^\epsilon(\xi) \mid \mathscr{F}_{t_{i-1}}] \underset{n,d\to\infty}{\longrightarrow} 0,\ \text{in prob}.\\
&\sum_{i=1}^n E[\chi_{i,d}^\epsilon(\xi)^2 \mid \mathscr{F}_{t_{i-1}}] \underset{n,d\to\infty}{\longrightarrow} 0,\ \text{in prob}.
\end{split}\]
\end{lem}
\begin{proof}
    We begin by proving the first claim.\ In the following,\ we abbreviate $E[\cdot \mid \mathscr{F}_{t_{i-1}}]$ by $E_i[\cdot]$.
    Then we have 
       \[
   \left|\sum_{i=1}^n E_i [\chi_{i,d}^\epsilon(\xi)] \right| \leq \frac{1}{\epsilon nd} \sum_{j=1}^d \sum_{i=1}^n |E_i [|I_{t_i}^{T_j}-I_{t_{i-1}}^{T_j}|^2/\Delta-\sigma_{t_{i-1}}^{T_j}(\theta)^2]|.
   \] 
   Let $j$ be fixed arbitrarily .\ If $T_j \leq t_{i-1}$,\ then by definition,\ the conditional expectation on the right-hand side is zero.\ Next,\ assume that $t_{i-1} < T_j \leq t_i$.\ In this case,\ Since $|T_j-t_{i-1}| \leq |t_i-t_{i-1}| =\Delta \leq 1$,\ the above conditional expectation can be evaluated as $M$.\ Suppose that $t_i < T_j $,\ since  
   $I_{t_i}^{T_j}-I_{t_{i-1}}^{T_j} =-\int_{t_{i-1}}^{t_i} V_{s}ds +\int_{t_{i-1}}^{t_i}\sigma_{s }^{T_j}(\theta) dW_s,
   $
the It\^o's isometry yields
   \[\begin{split}
       E\left|E_i \left[\frac{|I_{t_i}^{T_j}-I_{t_{i-1}}^{T_j}|^2}{\Delta}-\sigma_{t_{i-1}}^{T_j}(\theta)^2\right]\right|
       &\leq \frac{1}{\Delta}E\left|\int_{t_{i-1}}^{t_i} V_{s}ds\right|^2\\
       &+\frac{2}{\Delta}E\left[\left|\int_{t_{i-1}}^{t_i} V_{s}ds\right|^2\right]^{1/2}E\left[\int_{t_{i-1}}^{t_i} \sigma_s^{T_j}(\theta)^2ds\right]^{1/2} \\
       &+\frac{1}{\Delta}E\left[\int_{t_{i-1}}^{t_i} |\sigma_s^{T_j}(\theta)^2 -\sigma_{t_{i-1}}^{T_j}(\theta)^2|ds\right].
   \end{split}\]
   By Lemmas~\ref{OWImeorignOIAMOIMSdfkLQ} and \ref{EOrignrotihomoISADOISnrgoin},
   The right-hand side is bounded by $C \Delta^{1/2}$.
   Noting that there is at most one term with $t_{i-1} <T_j \leq t_i$,\ we obtain $E\left[\left|\sum_{i=1}^n E_i [\chi_{i,d}^\epsilon(\xi)] \right|\right] \leq C \Delta^{1/2}$ and hence we conclude that.

   We now turn to the proof of the second claim.\ Since
      \[\chi_{i,d}^\epsilon(\xi)^2 \leq \frac{2d}{\epsilon ^2n^2d^2 } \sum_{j=1}^d \left(\frac{|I_{t_i}^{T_j}-I_{t_{i-1}}^{T_j}|^4}{\Delta^2}+\sigma_{t_{i-1}}^{T_j}(\theta)^4\right),\]
      we arrive at the bound $E[\sum_{i=1}^n E_i[\chi_{i,d}^\epsilon(\xi)^2]] \leq C \Delta$ from Lemma~\ref{OWImeorignOIAMOIMSdfkLQ} and \ref{EOEITnhgoedtihmoIAOIDGHOIDFNhoiAMOIWS}.\ This completes the proof.
   \end{proof}

   \begin{lem}
       Under \hypref{EIrugnoIASNOsirgmoiNAOISmfgoiN},\ for all $\epsilon$,
       \[\sup_{\xi} \left|U_{n,d}^\epsilon(\xi)-\frac{1}{n} \sum_{i=1}^n U_{t_{i-1},d}^\epsilon(\xi)\right| \underset{n,d\to\infty}{\longrightarrow} 0,\ \text{in prob}.\]
   \end{lem}
   \begin{proof}
       To prove the claim,\ it suffices to show $\sup_{n,d}E[\sup_\xi |\nabla_\xi X_{n,d}^\epsilon(\xi)| ]<\infty$ and $X_{n,d}^\epsilon(\xi)\xrightarrow{p}0$ as $n,d \rightarrow 0$ for all $\xi$.
       
       First,\ it follows from Lemma~$9$ of \cite{genon1993estimation} and Lemmas~\ref{EOirgrthoimOASINOSdigdghNOSINR} that for any $\xi$ and $\epsilon$,\ $X_{n,d}^\epsilon(\xi)$ converges in probability to $0$ as $n,d \rightarrow \infty$.

       Next,\ we show $\sup_{n,d}E[\sup_\xi |\nabla_\xi X_{n,d}^\epsilon(\xi)| ]<\infty$.\ We obtain
       \[\nabla_\xi X_{n,d}^\epsilon(\xi) =\frac{1}{n}\sum_{i=1}^n \left(-\frac{\nabla_\xi|\hat{\sigma}_{t_{i-1}}(\xi)|^2}{(|\hat{\sigma}_{t_{i-1}}(\xi)|^2 +\epsilon d)^2}\right)(|\Delta I_i^d|^2 -|\sigma_{t_{i-1}}^d (\theta)|^2).\]
       Hence, as in the proof of Lemma~\ref{EOirgrthoimOASINOSdigdghNOSINR},
    \[
        E[\sup_\xi |\nabla_\xi X_{n,d}^\epsilon(\xi)|] \leq \frac{C}{ nd} \sum_{i=1}^n \sum_{j=1}^d  E\left[\frac{|I_{t_i}^{T_j}-I_{t_{i-1}}^{T_j}|^4}{\Delta^2}+\sigma_{t_{i-1}}^{T_j}(\theta)^4\right]^{1/2}\leq C.
   \]
  This establishes the assertion.
   \end{proof}

   \begin{lem}
   Under~\hypref{EIrugnoIASNOsirgmoiNAOISmfgoiN}, for all $\epsilon$,
       \[\sup_\xi \left|\frac{1}{n} \sum_{i=1}^n U_{t_{i-1},d}^\epsilon(\xi) - U_d^\epsilon(\xi)\right|\underset{n,d\to\infty}{\longrightarrow} 0,\ \text{in prob}.\]
   \end{lem}
   \begin{proof}
       We set $Y_{n,d}^\epsilon(\xi) \coloneqq \sum_{i=1}^n U_{t_{i-1},d}^\epsilon (\xi)/n$.\ To prove the assertion,\ it is enough to show\\ $\sup_{\xi,n,d}|\nabla_\xi Y_{n,d}^\epsilon(\xi)| <\infty$,\ $\sup_{\xi,d}|\nabla_\xi U_d^\epsilon(\xi)| <\infty$ and 
       $|Y_{n,d}^\epsilon(\xi) -U_d^\epsilon(\xi)| \xrightarrow{p}0$ as $n,d \rightarrow 0$ for all $\xi$.

       We first show that 
       \[\begin{split}
           &\nabla_\xi U_{t,d}^\epsilon (\xi) =\frac{\nabla_\xi |\hat{\sigma}_{t}(\xi)|^2}{|\hat{\sigma}_{t}(\xi)|^2+\epsilon d}-\frac{(|\sigma_{t}^d(\theta)|^2+\epsilon d)\nabla_\xi |\hat{\sigma}_{t}(\xi)|^2}{(|\hat{\sigma}_{t}(\xi)|^2+\epsilon d)^2}.
       \end{split}\]
       Thus, $|\nabla_\xi U_{t,d}^\epsilon (\xi)| \leq M$ holds for every $\xi,t,d$ and consequently $\sup_{\xi,n,d}|\nabla_\xi Y_{n,d}^\epsilon(\xi)| <\infty$,\\
       $\sup_{\xi,d}|\nabla_\xi U_d^\epsilon(\xi)| <\infty$ follows.

       Next,\ decomposing the difference into the logarithmic term and the fractional term, and using elementary estimates, we obtain
    \[|U_{t,d}^\epsilon(\xi) -U_{s,d}^\epsilon(\xi)| \leq \frac{M}{\sqrt{d}} |\hat{\sigma}_t(\xi)-\hat{\sigma}_s(\xi)| + \frac{M}{\sqrt{d}}|\sigma_t^d(\theta)-\sigma_s^d(\theta)|.\]
    Consequently, $E[|U_{t,d}^\epsilon(\xi)-U_{s,d}^\epsilon(\xi)|] \leq C|t-s|^{1/2}$ which in turn yields $E[|Y_{n,d}^\epsilon (\xi) - U_d^\epsilon(\xi)|] \leq C\Delta^{1/2}$.\ Hence, we obtain $|Y_{n,d}^\epsilon(\xi) -U_d^\epsilon(\xi)| \xrightarrow{p}0$ as $n,d \rightarrow 0$ for all $\xi$.
   \end{proof}

\begin{proof}[Proof of Theorem~\textnormal{\ref{EORigeothimoISNAOFISNrogsdifghaosfigneso}}]-- From the previous results and the triangle inequality,\ it suffices to show $\sup_{\xi}|U_d^\epsilon(\xi)-U^\epsilon(\xi)| \xrightarrow[]{p} 0$ as $d \rightarrow \infty$ for any $\epsilon$.\ Furthermore,\ it is enough to prove $\sup_{\xi,d}|\nabla_\xi U_d^\epsilon(\xi)| <\infty$,\ $\sup_\xi |\nabla_\xi U^\epsilon(\xi)|<\infty$ and $U_d^\epsilon(\xi) \xrightarrow{p} U^\epsilon(\xi)$ as $d \rightarrow 0$ for all $\xi$.\ However,\ $\sup_{\xi,d}|\nabla_\xi U_d^\epsilon(\xi)| <\infty$ has already been proved in the previous lemma and $\sup_\xi |\nabla_\xi U^\epsilon(\xi)|<\infty$ follows from the twice continuous differentiability of $U^\epsilon(\cdot)$.

        We now prove $U_d^\epsilon(\xi) \xrightarrow{p} U^\epsilon(\xi)$ as $d \rightarrow 0$ for all $\xi$.\ To begin with,\ by \hypref{EORigndgfhoimOAINWOirngdofgimOASINSOfiogm},\ for any $\xi$ and $t <1$,\ we have
        \[\frac{1}{d} |\sigma_t^d(\xi)|^2 =\frac{1}{d} \sum_{j=1}^d \sigma_t^{T_j}(\xi)^2 \rightarrow  \int_0^1 \sigma_t^x(\xi)^2 \mu(dx),\ d \rightarrow \infty\]
        with probability one.\ Moreover,  
        by Lipschitz continuity of k,
         $|\hat{\sigma}_t^{T_j}(\xi)-\sigma_t^{T_j}(\xi)|\leq M \max_{k} |T_k-T_{k-1}|$,
    and by \hypref{EORigndgfhoimOAINWOirngdofgimOASINSOfiogm} again,
    $||\hat{\sigma}_t(\xi)|^2-|\sigma_t^d(\xi)|^2|/d \leq M\max_{k} |T_k-T_{k-1}| \rightarrow 0,\ d \rightarrow \infty$.
    Hence,\ the dominated convergence theorem yields $U_d^\epsilon(\xi) \xrightarrow{p} U^\epsilon(\xi)$ as $d \rightarrow 0$.\ Thus, the assertion has been established.
\end{proof}

\begin{proof}[Proof of Corollary~\textnormal{\ref{EORginseoihmOIANOisrgmoiNAOiedrg}}]
    -- By the result of Theorem~\ref{EORigeothimoISNAOFISNrogsdifghaosfigneso},\ if the limit $U^\epsilon (\cdot)$ has a unique minimizer $\xi=\theta$ almost surely,\ then the claim follows.
    
    First,\ since the function $f(x)=\log(x)+y/x,\ x,y > 0$ has the unique minimizer $x=y$,\ it follows that $U^\epsilon(\xi) \geq U^\epsilon (\theta)$ and if $U^\epsilon(\xi)=U^\epsilon(\theta)$,\ then we obtain
    \begin{equation}
    \label{EROignetohgimOWINWEOrigmsdoiOSIFbgosifh}
    \int_0^1 \sigma_t^x(\xi)^2 \mu(dx) =\int_0^1 \sigma_t^x(\theta)^2 \mu(dx),\ t \text{-a.e.}\end{equation}

    Here we assume that the density of $\mu$ is given by a function $g>0$,\ and fix $t<1$ arbitrarily.\ Then,\ by Fubini's theorem,\ we obtain
    \[\begin{split}
        \int_0^1 \sigma_t^x(\xi)^2 \mu(dx)&=\int_0^{1-t}\int_0^{1-t} k(\xi,r)k(\xi,r')V_t^{t+r}V_t^{t+r'}\int_{t+r \vee r'}^1 g(u)dudrdr'\\
        &\eqqcolon \int_0^{1-t}\int_0^{1-t} k(\xi,r)k(\xi,r') K_t(r,r')drdr'.
    \end{split}\]
    Furthermore,\ if we set $h(t) \coloneqq k(\xi,t)+k(\theta,t)$,\ $\phi(t,r) \coloneqq V_t^{t+r}$ and $G(t,r) \coloneqq \int_{t+r}^1 g(u)du$,\ then we obtain
    \[\begin{split}
        &\int_0^1 \sigma_t^x(\xi)^2 \mu(dx)-\int_0^1 \sigma_t^x(\theta)^2 \mu(dx)
        =\int_0^{1-t} (k(\xi,r)-k(\theta,r))q(t,r)dr,\\
        &q(t,r)\coloneqq \phi(t,r)\left(G(t,r)\int_0^r h(r')\phi(t,r')dr'+\int_r^{1-t} h(r') \phi(t,r')G(t,r')dr'\right)
    \end{split}\]
    by symmetry of $K_t$.\ Moreover,\ since $g,h,\phi>0$,\ it follows that $q(t,\cdot) >0,\ \text{on}\ (0,1-t)$.

    Assume that $\xi \neq \theta$.\ Then from~\hypref{ERogienthoimOAINEOITRgdnofgnoINWSR},\ there exists $\delta>0$ such that $k(\xi,\cdot)-k(\theta,\cdot)>0$ (or $<0$) on $(0,\delta)$.\ Then,\ for all $t \in (1-\delta,1)$,\ we obtain
    \[\int_0^{1-t} (k(\xi,r)-k(\theta,r))q(t,r)dr >0\ (\text{or} <0),
     \]
    which contradicts \eqref{EROignetohgimOWINWEOrigmsdoiOSIFbgosifh}.\ Hence,\ we must have $\xi=\theta$.
\end{proof}
\section{Asymptotic Mixed Normality}
\label{OeigoainOISFDMgodfighoISB}
Unless explicitly mentioned otherwise,\ we will work under the assumptions~\hypref{EIrugnoIASNOsirgmoiNAOISmfgoiN},\\
\hypref{ERogienthoimOAINEOITRgdnofgnoINWSR} and \hypref{EORigndgfhoimOAINWOirngdofgimOASINSOfiogm}.\ We continue to regard $\theta$ as the true parameter value and assume that it is an interior point of $\Theta$.\ Let $\hat{\theta}_{n,d}^\epsilon$ be the minimal solution of $U_{n,d}^\epsilon(\cdot)$,\ and $\Omega_{n,d}^\epsilon$ be the set where $\hat{\theta}_{n,d}^\epsilon$ belongs to the interior of $\Theta$,\ so Corollary~\ref{EORginseoihmOIANOisrgmoiNAOiedrg} yields 
\begin{equation}
\label{EOrigneothinoIMASDOfisdnroginmoINASD}
    P(\Omega_{n,d}^\epsilon) \rightarrow 1.
\end{equation}

In this case,\ since $\nabla_\xi U_{n,d}^\epsilon(\hat{\theta}_{n,d}^\epsilon)=0$ on $\Omega_{n,d}^\epsilon$,\ it follows from Taylor's theorem that
\[    \begin{split}
0&=L_{n,d}^\epsilon+B_{n,d}^\epsilon\cdot \sqrt{n}(\hat{\theta}_{n,d}^\epsilon-\theta)\ \text{on}\ \Omega_{n,d}^\epsilon,\\
L_{n,d}^\epsilon &\coloneqq \sqrt{n}\nabla_\xi U_{n,d}^\epsilon(\theta),\ B_{n,d}^\epsilon \coloneqq \int_0^1 \nabla^2_\xi U_{n,d}^\epsilon(\theta+u(\hat{\theta}_{n,d}-\theta))du.
    \end{split}\]
    
In order to establish the asymptotic mixed normality,\ we first show the convergence of $L_{n,d}^\epsilon$ and $B_{n,d}^\epsilon$.
 As a first step,\ we prove that an approximation $\nabla^2_\xi U_{n,d}^\epsilon(\theta)$ of $B_{n,d}^\epsilon$ converges in probability to $B^\epsilon$~\eqref{ErogirethoimOIWNOEIrgndofighoIASN},\ where the $(\alpha,\beta)$-component of $B^\epsilon$ is restated as
\[(B^\epsilon)^{\alpha,\beta} = 4 \int_0^1 \frac{\left(\int_0^1 \partial_\alpha \sigma_t^x(\theta)\sigma_t^x(\theta)\mu(dx)\right)\left(\int_0^1 \partial_\beta \sigma_t^x(\theta)\sigma_t^x(\theta)\mu(dx)\right)}{\left(\int_0^1 \sigma_t^x(\theta)^2\mu(dx)+\epsilon\right)^2}dt.\]
The structure of the proof is essentially the same as that in \cref{02}.\ In what follows,\ we shall denote $\sigma_t^d(\theta)$ and $\hat{\sigma}_t(\theta)$ by $\sigma_t^d$ and $\hat{\sigma}_t$,\ respectively.
\begin{lem}
\label{EORIngotihmoSIFOIGNoaiNDF}
For any $\epsilon$,\ $\nabla_\xi^2U_{n,d}^\epsilon(\theta)$ converges to $B^\epsilon$ in probability as $n,d \rightarrow \infty$.
\end{lem}
\begin{proof}
    As the beginning of the proof,\ we define the sequence of stopping times $R_l$ by $R_l \coloneqq \inf \{t \in [0,1] \mid I_t^1 \geq l\}$,\ here we adopt the convention $\inf \emptyset =\infty$.\ In this case,\ if we set $\Omega_l \coloneqq \{R_l =\infty\}$,\ then $P(\Omega_l) \rightarrow 1$ holds.\ Furthermore,\ since $I_{t \wedge R_l}^T \leq I_{t \wedge R_l}^1 \leq l$ for all $t,T$,\ there exists a constant $C_l$ such that for any $t,d$ and $\alpha,\beta=1,\ldots,q$,\ $|\hat{\sigma}_{t \wedge R_l}| \leq C_l\sqrt{d}$,\ $|\partial_\alpha \hat{\sigma}_{t \wedge R_l}| \leq C_l\sqrt{d}$ and $|\partial_\alpha \partial_\beta \hat{\sigma}_{t \wedge R_l}| \leq C_l\sqrt{d}$\ holds in the same manner as in the proof of Lemma~\ref{OWImeorignOIAMOIMSdfkLQ}.

    Fix arbitrary $\epsilon$ and $\alpha,\beta=1,\ldots,q$. 
    We define 
    \[
        \chi_{n,d}^l \coloneqq \frac{1}{n}[\partial_\alpha \partial_\beta F_\epsilon(\theta,t_{i-1},\Delta I_i^d)-\partial_\alpha \partial_\beta U_{t_{i-1},d}^\epsilon(\theta)]1_{\{t_{i-1} <R_l\}}.
    \]
    If we abbreviate $E[\cdot \mid \mathscr{F}_{t_{i-1}}]$ as $E_i[\cdot]$,\ then,
    \[\left|\sum_{i=1}^n E_i[\chi_{i,d}^l]\right| \leq \frac{C_l}{ nd}\sum_{i=1}^n |E_i[|\Delta I_i^d|^2-|\sigma_{t_{i-1}}^d|^2]|,\quad(\chi_{i,d}^l)^2 \leq \frac{C_l}{n^2d^2}(|\Delta I_i^d|^2-|\sigma_{t_{i-1}}^d|^2)^2
   \]
    for a constant $C_l$,\ and hence,\ as in the proof of Lemma~\ref{EOirgrthoimOASINOSdigdghNOSINR},\ we obtain \\
    $\sum_{i=1}^n E_i[\chi_{n,d}^l] \xrightarrow{p}0$ and $\sum_{i=1}^n E_i[(\chi_{n,d}^l)^2] \xrightarrow{p}0$,\ as $n,d \rightarrow \infty$.\ Therefore,\ by Lemma~$9$ of \cite{genon1993estimation},\ we obtain $ \sum_{i=1}^n \chi_{n,d}^l \xrightarrow{p}0$ as $n,d \rightarrow \infty$.\ Moreover,\ since $P(\Omega_l) \rightarrow 1$,\ if we set $Y_{n,d}^\epsilon \coloneqq \sum_{i=1}^n \partial_\alpha \partial_\beta U_{t_{i-1},d}^\epsilon(\theta)/n$,\ then it follows that $|\partial_\alpha \partial_\beta U_{n,d}^\epsilon(\theta)-Y_{n,d}^\epsilon| \xrightarrow{p}0$ as $n,d \rightarrow \infty$.
    
    Next,\ we shall prove $|Y_{n,d}^\epsilon -\partial_\alpha \partial_\beta U_d^\epsilon(\theta)|\xrightarrow{p}0$ as $n,d \rightarrow \infty$.
     In what follows, we use the same notation for the random
     variable $M$ and the constant $C$ appearing in the estimates.
    We first note that
    \[\begin{split}
    &|\partial_\alpha\partial_\beta U_{t,d}^\epsilon (\theta) - \partial_\alpha\partial_\beta U_{s,d}^\epsilon (\theta)| \\
    &\leq \frac{M}{\sqrt{d}} (|\hat{\sigma}_t-\hat{\sigma}_s|+|\partial_\alpha\hat{\sigma}_t-\partial_\alpha\hat{\sigma}_s| + |\partial_\beta\hat{\sigma}_t-\partial_\beta\hat{\sigma}_s| + |\partial_\alpha\partial_\beta\hat{\sigma}_t-\partial_\alpha\partial_\beta\hat{\sigma}_s| + | \sigma_t^d - \sigma_s^d|).
    \end{split}\]
Therefore, we have $E[|\partial_\alpha \partial_\beta U_{t,d}^\epsilon(\theta)-\partial_\alpha \partial_\beta U_{s,d}^\epsilon(\theta)|] \leq C|t-s|^{1/2}$.
   which in turn yields
    $E[|Y_{n,d}^\epsilon -\partial_\alpha \partial_\beta U_d^\epsilon(\theta)|] \leq C\Delta^{1/2}$.
    Consequently, $|Y_{n,d}^\epsilon -\partial_\alpha \partial_\beta U_d^\epsilon(\theta)| \xrightarrow[]{p} 0$ as $n,d \rightarrow \infty$.
       
    Finally,\ we show that $\partial_\alpha \partial_\beta U_d^\epsilon(\theta) \xrightarrow{p} (B^\epsilon)^{\alpha,\beta}$,\ as $d \rightarrow \infty$.\ 
First,\ as shown in the proof of Theorem~\ref{EORigeothimoISNAOFISNrogsdifghaosfigneso},
it follows that 
$||\hat{\sigma}_t|^2-|\sigma_t^d|^2|/d \rightarrow 0$, as $d \rightarrow \infty$.
    By Lemma~\ref{OWImeorignOIAMOIMSdfkLQ} and \hypref{EORigndgfhoimOAINWOirngdofgimOASINSOfiogm},\ we have
        \[\frac{1}{d} (\partial_\alpha\sigma_t^d)^\mathsf{T} \sigma_t^d =\frac{1}{d} \sum_{j=1}^d \partial_\alpha \sigma_t^{T_j} \sigma_t^{T_j}\rightarrow  \int_0^1 \partial_\alpha \sigma_t^x \sigma_t^x \mu(dx),\ d \rightarrow \infty\]
        with probability one.\ Moreover,\ similarly to the proof of Theorem~\ref{EORigeothimoISNAOFISNrogsdifghaosfigneso},\ we obtain$|(\partial_\alpha \hat{\sigma}_t)^\mathsf{T} \hat{\sigma}_t -(\partial_\alpha \sigma_t^d )^\mathsf{T} \sigma_t^d|/d\leq M\max_k|T_k-T_{k-1}| \rightarrow 0,\ d \rightarrow \infty.$
        Hence,\ the dominated convergence theorem yields $\partial_\alpha \partial_\beta U_d^\epsilon(\theta) \xrightarrow{p} (B^\epsilon)^{\alpha,\beta}$,\ as $d \rightarrow \infty$.\ This completes the proof of $\partial_\alpha \partial_\beta U_{n,d}^\epsilon(\theta) \xrightarrow{p} (B^\epsilon)^{\alpha,\beta}$,\ as $n,d \rightarrow \infty$. 
\end{proof}

Next,\ we verify that $B_{n,d}^\epsilon$ itself converges.

\begin{lem}
For any $\epsilon$,\ $B_{n,d}^\epsilon$ converges to $B^\epsilon$ in probability as $n,d \rightarrow \infty$.
\end{lem}
\begin{proof}
Let $\epsilon$ be arbitrary.\ It suffices to show that $N_{n,d}(a_{n,d}) \xrightarrow[]{p}0$ as $n,d \rightarrow \infty$ for any sequence $a_{n,d} \downarrow 0$,\ where $N_{n,d}(a) \coloneqq \sup_{\zeta:|\zeta| \leq a} |\nabla^2_\xi U_{n,d}^\epsilon(\theta+\zeta)-\nabla^2_\xi U_{n,d}^\epsilon(\theta)|$,\ from the previous result.

    First,\ define the stopping times $R_l$ and the events $\Omega_l$ in the same way as in the proof of Lemma~\ref{EORIngotihmoSIFOIGNoaiNDF}.\ Moreover,\ take $a >0$ sufficiently small,\ and let $\zeta \in \Theta$ be arbitrary with $|\zeta| \leq a$.\ Then,\ since $I_t^T \leq l$ for all $t,T$ on $\Omega_l$,\ we have
    \[|\hat{\sigma}_t(\theta+\zeta)-\hat{\sigma}_t(\theta)| \leq l\sqrt{d}\sup_{\zeta:|\zeta | \leq a ,t }|k(\theta+\zeta,t)-k(\theta,t)|,\ \text{on}\ \Omega_l.\]
    Here,\ by the uniform continuity of $k$,\ the right-hand side converges to $0$ as $a \downarrow 0$.\ Similarly,\ there exists a function $\eta(a) \rightarrow 0$ as $a \downarrow 0$ such that,\ for any $t,\alpha,\beta$,\ $|\hat{\sigma}_t(\theta+\zeta)-\hat{\sigma}_t(\theta)| \leq l\sqrt{d}\eta(a)$,\ $|\partial_\alpha\hat{\sigma}_t(\theta+\zeta)-\partial_\alpha\hat{\sigma}_t(\theta)| \leq l\sqrt{d}\eta(a)$ and $|\partial_\alpha\partial_\beta\hat{\sigma}_t(\theta+\zeta)-\partial_\alpha\partial_\beta\hat{\sigma}_t(\theta)| \leq l\sqrt{d}\eta(a)$ holds on $\Omega_l$.\ In this case, we obtain the estimate with  a constant $C_l$,
    \[
        E[\sup_{\zeta:|\zeta| \leq a}|\partial_\alpha\partial_\beta U_{n,d}^\epsilon(\theta+\zeta)-\partial_\alpha\partial_\beta U_{n,d}^\epsilon(\theta)|1_{\Omega_l}]\leq \frac{C_l \eta(a)}{nd} \sum_{i=1}^n (E[|\Delta I_i^d|^2]+d),
    \]
     in the same way as in the proof of Lemma~\ref{EORIngotihmoSIFOIGNoaiNDF}.\ Furthermore,\ by Lemma~\ref{EOEITnhgoedtihmoIAOIDGHOIDFNhoiAMOIWS},\ there exists a constant $C_l$, \[E[\sup_{\zeta:|\zeta| \leq a}|\partial_\alpha\partial_\beta U_{n,d}^\epsilon(\theta+\zeta)-\partial_\alpha\partial_\beta U_{n,d}^\epsilon(\theta)|1_{\Omega_l}] \leq C_l \eta(a).\] Finally,\ we obtain $N_{n,d}(a_{n,d})1_{\Omega_l} \xrightarrow[]{p}0$ as $n,d \rightarrow \infty$ for any sequence $a_{n,d} \downarrow 0$ and since $P(\Omega_l) \rightarrow 1$,\ the claim follows.
\end{proof}
Let us define the approximation $L_{n,d}^{\epsilon\ \prime}$ of $L_{n,d}^\epsilon$ as 
\[L_{n,d}^{\epsilon\ \prime}  \coloneqq \frac{1}{\sqrt{n}}\sum_{i=1}^n\nabla_\xi F_\epsilon(\theta,t_{i-1},\Delta J_i^d),\ \Delta J_i^d \coloneqq \frac{\sigma^d_{t_{i-1}}(\theta)}{\sqrt{\Delta}}(W_{t_i}-W_{t_{i-1}}).\]
Then,\ the following holds.
\begin{lem}
\label{OiegndogihmoiNAOSIDmsr}
    For any $\epsilon$,\ $L_{n,d}^\epsilon-L_{n,d}^{\epsilon\ \prime}$ converges to $0$ in probability as $n,d \rightarrow 0$.
\end{lem}
\begin{proof}
   We define the stopping times $R_l$ and the events $\Omega_l$ in the same way as in the proof of Lemma~\ref{EORIngotihmoSIFOIGNoaiNDF}.\ Fix $\epsilon,\alpha$ arbitrarily and set
   \[\chi_{i,d}^l \coloneqq \frac{1}{\sqrt{n}} [\partial_\alpha F_\epsilon(\theta,t_{i-1},\Delta I_i^d )-\partial_\alpha F_\epsilon(\theta,t_{i-1},\Delta J_i^d)]1_{\{t_{i-1}<R_l\}}.\]
   Hereafter,\ we simply denote $E[\cdot \mid\mathscr{F}_{t_{i-1}}]$ by $E_i[\cdot]$.\ 
   
   We first show $\sum_{i=1}^n E_i[\chi_{i,d}^l] \xrightarrow[]{p}0$ as $n,d \rightarrow \infty$.\ Let us define $g_i(x) \coloneqq \partial_\alpha F_\epsilon(\theta,t_{i-1},x)$.
   Then by Taylor's formula,\ we get
 \begin{equation}
 \label{EOrtigndoghimoIASsoigfndsfogim}
      g_i(\Delta I_i^d) -g_i(\Delta J_i^d)=\nabla g_i(\Delta J_i^d)^\mathsf{T}(\Delta I_i^d -\Delta J_i^d) -\frac{2(\partial_\alpha \hat{\sigma}_{t_{i-1}})^\mathsf{T}\hat{\sigma}_{t_{i-1}}}{(|\hat{\sigma}_{t_{i-1}}|^2+\epsilon d)^2}|\Delta I_i^d -\Delta J_i^d|^2.
 \end{equation}
 Here we set $\Delta W_i \coloneq W_{t_i}-W_{t_{i-1}}$.\ In what follows,\ constants used the estimates and those depending on $l$ will be denoted by the same symbols $C$ and $C_l$.\ Then,\ we have
 \[
    \left|\sum_{i=1}^n E_i[\chi_{i,d}^l]\right|\leq \frac{1}{\sqrt{n}} \cdot \frac{C_l}{d} \sum_{i=1}^n (|E_i[(\Delta J_i^d)^\mathsf{T}(\Delta I_i^d -\Delta J_i^d)]|1_{\{t_{i-1}<R_l\}}+E_i[|\Delta I_i^d -\Delta J_i^d|^2]).
 \]
By Lemmas~\ref{OWImeorignOIAMOIMSdfkLQ} and \ref{EOrignrotihomoISADOISnrgoin}, for any $p \geq 2$, we obtain
\begin{equation}
\label{OIEnrtgoidfghmoIASBDEOFISr}
    \sum_{i=1}^nE[|\Delta I_i^d -\Delta J_i^d|^p] \leq Cd^{p/2} (1+\Delta^{p/2-1}).
\end{equation}
Furthermore, we obtain 
     \[
         \sum_{i=1}^n |E_i[ (\Delta J_i^d)^\mathsf{T}(\Delta I_i^d -\Delta J_i^d)]|1_{\{t_{i-1}<R_l\}}\leq \frac{C_l}{\Delta}\sum_{j=1}^d \sum_{i=1}^n[A_{i} + A_{i,j}'] ,
     \]
     where
     \[A_{i} \coloneqq \left|E_i\left[\Delta W_i \int_{t_{i-1}}^{t_i} V_sds\right]\right|,\quad A_{i,j}' \coloneqq \left|E_i\left[\Delta W_i\int_{t_{i-1}}^{t_i} (\sigma_{s}^{T_j}-\sigma_{t_{i-1}}^{T_j})dW_s\right]\right|.\]
     We observe $E[A_{i}] \leq \Delta^{3/2} \tau(\Delta)$, where $\tau(\Delta) \coloneqq E[\sup_{|t-s|\leq \Delta} |V_t-V_s|^2]^{1/2}$. By the uniform continuity of $V$ and the dominated convergence theorem,\ we obtain that $\tau(\Delta) \rightarrow 0$ as $\Delta \rightarrow 0$.\ From Lemma~\ref{EOrignrotihomoISADOISnrgoin},
     \[\sigma_s^{T_j} - \sigma_{t_{i-1}}^{T_j} = u_s + v_s + \phi_{t_{i-1}}^{T_j}(W_s - W_{t_{i-1}}),\ u_s \coloneqq \int_{t_{i-1}}^s b_r^{T_j} dr,\ v_s \coloneqq \int_{t_{i-1}}^s (\phi_r^{T_j}-\phi_{t_{i-1}}^{T_j})dW_r. \]
     It follows that $E[A_{i,j}']\leq C\Delta^2$.
     Therefore,\ we have
     \[E\left[\sum_{i=1}^n |E_i[(\Delta J_i^d)^\mathsf{T}(\Delta I_i^d -\Delta J_i^d)]|1_{\{t_{i-1}<R_l\}}\right] \leq C_ld(1+\sqrt{n}\tau(\Delta)).\]
     Thus,
     \[
         E\left[\left|\sum_{i=1}^nE_i[\chi_{i,d}^l]\right|\right] \leq \frac{1}{\sqrt{n}} \cdot \frac{C_l}{d}(d(1+\sqrt{n}\tau(\Delta)) +d) \leq C_l(\Delta^{1/2} +\tau(\Delta)) \rightarrow 0,\ n,d \rightarrow \infty.
      \]

     We turn to the proof of $\sum_{i=1}^n E_i[(\chi_{i,d}^l)^2] \xrightarrow[]{p}0$,\ as $n,d \rightarrow \infty$.\ From \eqref{EOrtigndoghimoIASsoigfndsfogim},\ we have
     \[(\chi_{i,d}^l)^2 \leq \frac{C_l}{d}(\Delta W_i)^2 |\Delta I_i^d-\Delta J_i^d|^2 +\frac{C_l }{nd^2} |\Delta I_i^d -\Delta J_i^d|^4 .\]
     Therefore, by \eqref{OIEnrtgoidfghmoIASBDEOFISr}, $E\left[\sum_{i=1}^n E_i[(\chi_{i,d}^l)^2]\right] \leq C_l\Delta^{1/2} \rightarrow 0,\ n,d \rightarrow \infty$.

     From the above,\ $\sum_{i=1}^n \chi_{i,d}^l \xrightarrow[]{p}0$,\ as $n,d \rightarrow \infty$.\ Finally,\ since $P(\Omega_l) \rightarrow 1$,\ the assertion follows.
\end{proof}

Let $(\mathscr{G}_t)$ denote the filtration generated by the standard Brownian motion $W$.\ We shall show that $L_{n,d}^\epsilon$ $\mathscr{G}_1$-stably converges in law to $L^\epsilon$,\ where $L^\epsilon$ is defined on an extension of the space $\Omega$ and is centered Gaussian,\ conditionally on $\mathscr{G}_1$,\ with covariance matrix $D^\epsilon$~\eqref{ErogirethoimOIWNOEIrgndofighoIASN},\ where the $(\alpha,\beta)$-component of $D^\epsilon$ is restated as
\[ \begin{split}
    (D^\epsilon)^{\alpha,\beta} &= \int_0^1z_t^\epsilon \left(\int_0^1 \partial_\alpha \sigma_t^x(\theta)\sigma_t^x(\theta)\mu(dx)\right)\left(\int_0^1 \partial_\beta \sigma_t^x(\theta)\sigma_t^x(\theta)\mu(dx)\right) dt,\\
    z_t^\epsilon&= \frac{8\left(\int_0^1 \sigma_t^x(\theta)^2\mu(dx)\right)^2}{\left(\int_0^1 \sigma_t^x(\theta)^2\mu(dx)+\epsilon\right)^4}.
\end{split}\]
\begin{lem}
Assume additionally that \hypref{OEIrgnrsotihymoIANOisrfg} holds.
    For all $\epsilon$,\ $L_{n,d}^\epsilon$ $\mathscr{G}_1$-stably converges in law to $L^\epsilon$.
\end{lem}
\begin{proof}
    First,\ by Lemma~\ref{OiegndogihmoiNAOSIDmsr},\ it suffices to show the stable convergence of $L_{n,d}^{\epsilon\ \prime}$ instead of $L_{n,d}^\epsilon$.\ Fix $\epsilon$ arbitrarily and define $\mathscr{G}_{t_{i}}$-measurable fourth moment integrable $\mathbb{R}^q$-valued random vectors $\psi_{i,d}$,\ where the $\alpha$-component of $\psi_{i,d}$ is\\ $(1/\sqrt{n}) \partial_\alpha F_\epsilon (\theta,t_{i-1},\Delta J_i^d)$.\ We represent $E[\cdot \mid \mathscr{G}_{t_{i-1}}]$ in terms of $E_i[\cdot]$ as follows.\ Then,\ according to \cite{Jacod1997},\ it suffices to show the following in order to prove that $L_{n,d}^{\epsilon\ \prime}$ $\mathscr{G}_1$-stably converges to $L^\epsilon$.
    \begin{enumerate}
        \item $\sup_t \left|\sum_{i=1}^{[nt]} E_i[\psi_{i,d}]\right| \xrightarrow[]{p}0,\ n,d \rightarrow \infty$,
        \item For all $t,\alpha,\beta$,
        $\sum_{i=1}^{[nt]} (E_i[\psi_{i,d}^\alpha \psi_{i,d}^\beta]-E_i[\psi_{i,d}^\alpha]E_i[\psi_{i,d}^\beta]) \xrightarrow[]{p} (D^\epsilon)^{\alpha,\beta}(t),\ n,d \rightarrow \infty$,
        where 
        \[ \begin{split}
    (D^\epsilon)^{\alpha,\beta}(t) &\coloneqq \int_0^t z_s^\epsilon\left(\int_0^1 \partial_\alpha \sigma_s^x\sigma_s^x\mu(dx)\right)\left(\int_0^1 \partial_\beta \sigma_s^x\sigma_s^x\mu(dx)\right)ds,\\
    z_t^\epsilon &\coloneqq \frac{8\left(\int_0^1 (\sigma_t^x)^2\mu(dx)\right)^2}{\left(\int_0^1 (\sigma_t^x)^2\mu(dx)+\epsilon\right)^4},
\end{split}\]
\item For every $t$,\ $\sum_{i=1}^{[nt]} E_i[ \psi_{i,d} (W_{t_i}-W_{t_{i-1}})]\xrightarrow[]{p}0,\ n,d \rightarrow \infty$,
\item $\sum_{i=1}^n E_i[|\psi_{i,d}|^4] \xrightarrow[]{p}0,\ n,d \rightarrow \infty$,
\item For any $t$,\ $\sum_{i=1}^{[nt]}E_i[\psi_{i,d}(N_{t_i}-N_{t_{i-1}})] \xrightarrow[]{p}0,\ n,d \rightarrow \infty$,\ where $N$ is bounded $(\mathscr{G}_t)$-martingale orthogonal to $W$.
    \end{enumerate}
    We begin with (i). As confirmed in the proof of Theorem~\ref{EORigeothimoISNAOFISNrogsdifghaosfigneso},\ there exists a random variable $M<\infty$ such that for all $\alpha$,
    \[
        |E_i[\psi_{i,d}^\alpha]| \leq \frac{1}{\sqrt{n}}\cdot \frac{2|(\partial_\alpha \hat{\sigma}_{t_{i-1}})^\mathsf{T} \hat{\sigma}_{t_{i-1}}|}{(|\hat{\sigma}_{t_{i-1}}|^2+\epsilon d)^2}||\hat{\sigma}_{t_{i-1}}|^2-|\sigma_{t_{i-1}}^d|^2|\leq \frac{M }{\sqrt{n}} \max_{k}|T_k-T_{k-1}|,
    \]
    and hence $\sup_t \left|\sum_{i=1}^{[nt]} E_i[\psi_{i,d}]\right| \leq M\sqrt{n} \max_k|T_k-T_{k-1}|$.\ Therefore,\ by \hypref{OEIrgnrsotihymoIANOisrfg},\ this converges to zero as $n,d \rightarrow \infty$.

    To show (ii),\ take arbitrary $\alpha$ and $\beta$.\ We first observe that
    \[
        E_i[\psi_{i,d}^\alpha \psi_{i,d}^\beta]-E_i[\psi_{i,d}^\alpha]E_i[\psi_{i,d}^\beta]
        =\frac{8}{n} \cdot A_{\alpha,t_{i-1}}A_{\beta,t_{i-1}}|\sigma_{t_{i-1}}^d|^4,\quad 
        A_{\alpha,t} \coloneqq \frac{(\partial_\alpha \hat{\sigma}_{t})^\mathsf{T} \hat{\sigma}_{t}}{(|\hat{\sigma}_{t}|^2+\epsilon d)^2}.
    \]
    Here we define $\chi_{d}(t) \coloneqq 8A_{\alpha,t}A_{\beta,t}|\sigma_t^d|^4$.\ Fix $t$ arbitrarily.\ First,\ we show\\ $|\sum_{i=1}^{[nt]}\chi_{d}(t_{i-1})/n -\int_0^t\chi_{d}(s)ds| \xrightarrow[]{p}0$ as $n,d \rightarrow \infty$.\ To begin with,\ we observe that
    \[
        \left|\sum_{i=1}^{[nt]}\chi_{d}(t_{i-1})/n -\int_0^t\chi_{d}(s)ds\right|\leq \sum_{i=1}^{[nt]}\int_{t_{i-1}}^{t_i}|\chi_{d}(t_{i-1})-\chi_{d}(s)|ds+\int_{[nt]/n}^t|\chi_{d}(s)|ds.
    \]
    In what follows,\ we use the same symbol $M$ for random variables employed in the estimates,\ all of which are integrable to the required powers.\ Since $|\chi_{d}| \leq M$,\ the second term on the right-hand side of the above expression converges to zero as $n,d \rightarrow \infty$.\ Next,\ we have
    \[
        |\chi_{d}(t)-\chi_{d}(s)|\leq M \left[d |A_{\alpha,t}-A_{\alpha,s}| +d |A_{\beta,t}-A_{\beta,s}|+\frac{1}{\sqrt{d}}|\sigma_t^d-\sigma_s^d|\right].
    \]
    Moreover,\ we have
    \[
        |A_{\alpha,t}-A_{\alpha,s}|
        \leq \frac{M}{d^{3/2}}(|\hat{\sigma}_{t}-\hat{\sigma}_{s}| +|\partial_\alpha \hat{\sigma}_{t}-\partial_\alpha \hat{\sigma}_{s}|).
    \]
    From the above,\ it follows from Lemma~\ref{EOEITnhgoedtihmoIAOIDGHOIDFNhoiAMOIWS} that there exists a constant $C$ such that $E[|\chi_{d}(t)-\chi_{d}(s)|] \leq C|t-s|^{1/2}$,\ and hence $|\sum_{i=1}^{[nt]}\chi_{d}(t_{i-1})/n -\int_0^t\chi_{d}(s)ds| \xrightarrow[]{p}0$ as $n,d \rightarrow \infty$.\ As established in the proofs of Theorem~\ref{EORigeothimoISNAOFISNrogsdifghaosfigneso} and Lemma~\ref{EORIngotihmoSIFOIGNoaiNDF},\ we have $\frac{1}{d} |\sigma_t^d|^2 \rightarrow \int_0^1 (\sigma_t^x)^2\mu(dx)$,\ $\frac{1}{d} |\hat{\sigma}_t|^2 \rightarrow \int_0^1 (\sigma_t^x)^2\mu(dx)$ and $\frac{1}{d} (\partial_\alpha \hat{\sigma}_t)^\mathsf{T}\hat{\sigma_t} \rightarrow  \int_0^1 \partial_\alpha \sigma_t^x \sigma_t^x \mu(dx)$,\ and hence,\ by the dominated convergence theorem,\ $\int_0^t\chi_{d}(s)ds \xrightarrow[]{p} (D^\epsilon)^{\alpha,\beta}(t)$,\ as $d \rightarrow \infty$.\ Thus,\ (ii) is established.

    (iii) is clear from the definition.\ We now prove (iv).\ Since
    \[E_i[|\psi_{i,d}|^4] \leq \frac{C}{n^2} \sum_\alpha \sum_\beta A_{\alpha,t_{i-1}}^2 A_{\beta,t_{i-1}}^2 (|\hat{\sigma}_{t_{i-1}}|^8+|\sigma_{t_{i-1}}^d|^8)\]
    for another constant $C$ and $A_{\alpha,t_{i-1}}^2 A_{\beta,t_{i-1}}^2 (|\hat{\sigma}_{t_{i-1}}|^8+|\sigma_{t_{i-1}}^d|^8) \leq M$,\ we have $\sum_{i=1}^n E_i[|\psi_{i,d}|^4] \leq Mq^2/n \rightarrow 0$,\ as $n,d \rightarrow \infty$.

    It remains to prove (v).\ Let $N$ be any bounded $(\mathscr{G}_t)$-martingale orthogonal to $W$.\ Since $(\mathscr{G}_t)$ is generated by Brownian motion $W$,\ it follows from the It\^{o} representation theorem that $N$ is constant.\ Thus (v) is satisfied.
\end{proof}

We now discuss the invertiblity of $B^\epsilon$.\ Suppose that
$x^\mathsf{T} B^\epsilon x = \int_0^1 y_t^\epsilon (x^\mathsf{T}e_t)^2 dt =0$.
Then it follows that
\begin{equation}
\label{OEInrgodimoiAOIDndsorigdseotriNA}
    x^\mathsf{T} e_t =0,\ t \text{-a.e.}
\end{equation}
Fix arbitrary $t<1$.\ Arguing in the same way as in the proof of Corollary~\ref{EORginseoihmOIANOisrgmoiNAOiedrg} (and using the same notation as there), we obtain
\[\begin{split}x^\mathsf{T} e_t &= \int_0^{1-t}\sum_\alpha x_\alpha \partial_\alpha k(\theta,r) \int_0^{1-t}k(\theta,r')K_t(r,r')dr'dr\\
&=\int_0^{1-t}\sum_\alpha x_\alpha\partial_\alpha  k(\theta,r) q'(t,r)dr,\\
q'(t,r) 
&\coloneqq \phi(t,r)\left(G(t,r)\int_0^r k(r')\phi(t,r')dr' + \int_r^{1-t} k(r') \phi(t,r') G(t,r')dr'\right).
\end{split}\]
Furthermore,\ we have $q'(t,\cdot) >0$ on $(0,1-t)$.

Now suppose that $x \neq 0$.\ Then, by \hypref{ErogisdotihnoiSMDOGFIsnrofgidsnfoginOIW},\ there exists $\delta >0$ such that \\
$\sum_\alpha x_\alpha \partial_\alpha k(\theta,\cdot) >0\ (\text{or} <0)$\ on\ $(0,\delta)$.\ Taking any $t \in (1-\delta,1)$,\ we obtain
\[\int_0^{1-t}\sum_\alpha x_\alpha \partial_\alpha k(\theta,r) q'(t,r)dr >0\ (\text{or} <0),\]
which contradicts \eqref{OEInrgodimoiAOIDndsorigdseotriNA}.\ Therefore, we conclude that $x = 0$.\ Similarly,\ $D^\epsilon$ is also invertible.

\begin{proof}[Proof of Theorem~\textnormal{\ref{OWIRgnosdifgnoIAMOSIdf}}]
    -- Let $\Omega_{n,d}^{\epsilon\ \prime} \coloneqq \{\omega \in \Omega_{n,d}^\epsilon \mid B_{n,d}^\epsilon(\omega)\ \text{is invertible}\}$,\ and $B_{n,d}^{\epsilon\ \prime} =B_{n,d}^\epsilon$ on $\Omega_{n,d}^{\epsilon\ \prime}$,\ and $B_{n,d}^{\epsilon\ \prime}=I_q$ (the $q \times q$ identity matrix) on $(\Omega_{n,d}^{\epsilon\ \prime})^c$.\ Then \eqref{EOrigneothinoIMASDOfisdnroginmoINASD} yields $P(\Omega_{n,d}^{\epsilon\ \prime}) \rightarrow 1$ and $(B_{n,d}^{\epsilon\ \prime})^{\ -1} \xrightarrow[]{p}(B^\epsilon)^{-1}$.\ Since $\mathscr{F}^I \subset \mathscr{G}_1$,\ $L_{n,d}^\epsilon$ $\mathscr{F}^I$-stably converges in law to $L^\epsilon$.\ Furthermore,\ from integration by parts we obtain
    $\sigma_t^x =I_t^x k(\theta,x-t)-\int_t^xI_t^u \partial_tk(\theta,u-t)du$,
    hence $B^\epsilon$ and $D^\epsilon$ are $\mathscr{F}^I$-measurable and $-(B_{n,d}^{\epsilon\ \prime})^{-1}L_{n,d}^\epsilon$ $\mathscr{F}^I$-stably converges in law to the claim $S^\epsilon=-(B^\epsilon)^{-1}L^\epsilon$.\ Finally,\ in view of $\sqrt{n}(\hat{\theta}_{n,d}-\theta) =-(B_{n,d}^{\epsilon\ \prime})^{\ -1}L_{n,d}^\epsilon$ on $\Omega_{n,d}^{\epsilon\ \prime}$,\ the desired result follows.
\end{proof}

\begin{proof}[Proof of Corollary~\textnormal{\ref{EEGRTOAINoinoidsnSOAI}}]
--Take an arbitrary $\epsilon >0$.\ Since Corollary~\ref{EORginseoihmOIANOisrgmoiNAOiedrg} holds, it suffices to show that $\hat{B}_{n,d}^\epsilon(\xi)\rightarrow B^\epsilon(\xi)$ and $\hat{D}_{n,d}^\epsilon(\xi)\rightarrow D^\epsilon(\xi)$ $\xi$-uniformly in probability as $n,d \rightarrow \infty$.\ To this end,\ define $B_d^\epsilon(\xi)$ and $D_d^\epsilon(\xi)$ by
\[\begin{split}
    B_d^\epsilon(\xi) \coloneqq \int_0^1 \hat{y}_{t,d}^\epsilon(\xi) \hat{e}_{t,d}(\xi)\hat{e}_{t,d}(\xi)^\mathsf{T}dt, D_d^\epsilon(\xi) \coloneqq \int_0^1 \hat{z}_{t,d}^\epsilon(\xi) \hat{e}_{t,d}(\xi)\hat{e}_{t,d}(\xi)^\mathsf{T}dt
\end{split}\]
and proceed step by step. 

First,\ we show that $\sup_{\xi}|\hat{B}_{n,d}^\epsilon(\xi) -B_d^\epsilon(\xi)|\xrightarrow[n,d \rightarrow \infty]{p} 0$. To prove this, it suffices to show that for each $\alpha,\beta =1,\ldots,q$,\ $\hat{B}_{n,d}^\epsilon(\xi)^{\alpha,\beta}-B_d^\epsilon(\xi)^{\alpha,\beta}$ converges to $0$ in probability at each $\xi$,\ and that $\sup_{\xi,n,d}|\nabla_\xi \hat{B}_{n,d}^\epsilon(\xi)^{\alpha,\beta}| <\infty$ and $\sup_{\xi,d} |\nabla_\xi B_d^\epsilon(\xi)^{\alpha,\beta} | <\infty$.\ The argument for $\sup_\xi|\hat{D}_{n,d}^\epsilon(\xi)-D_d^\epsilon(\xi)| \xrightarrow[n,d \rightarrow \infty]{p} 0$ is similar. By Lemma~\ref{OWImeorignOIAMOIMSdfkLQ},\ there exists a random variable $M$ such that for each $\xi,t,d$ and $\alpha$,\ we obtain the estimates $|\hat{y}_{t,d}^\epsilon(\xi)| \leq 4/\epsilon^2$,\ $|z_{t,d}^\epsilon(\xi)| \leq M$ and $|e_{t,d}^\epsilon(\xi)^\alpha|  \leq M$.\ Furthermore,\ for each $\gamma$,
$|\partial_\gamma \hat{y}_{t,d}^\epsilon(\xi)| \leq M$,
    $|\partial_\gamma \hat{z}_{t,d}^\epsilon(\xi)| \leq M$, and
    $|\partial_\gamma \hat{e}_{t,d}(\xi)^\alpha|\leq M$.
Therefore, we have 
$\sup_{\xi,n,d}|\nabla_\xi \hat{B}_{n,d}^\epsilon(\xi)^{\alpha,\beta}| <\infty$ and $\sup_{\xi,d}|\nabla_\xi B_d^\epsilon(\xi)^{\alpha,\beta}| <\infty$. Similarly,\ we obtain $\sup_{\xi,n,d}|\nabla_\xi \hat{D}_{n,d}^\epsilon(\xi)^{\alpha,\beta}| <\infty$ and $\sup_{\xi,d}|\nabla_\xi D_d^\epsilon(\xi)^{\alpha,\beta}| <\infty$.
Again by Lemma~\ref{OWImeorignOIAMOIMSdfkLQ},\ there exists a random variable $M$ such that for any $\xi,t,s,d$ and $\alpha$,
\[\begin{split}
    &|\hat{y}_{t,d}^\epsilon(\xi) - \hat{y}_{s,d}^\epsilon(\xi)| \leq \frac{M}{\sqrt{d}}|\hat{\sigma}_t(\xi)-\hat{\sigma}_s(\xi)|,\quad |\hat{z}_{t,d}^\epsilon(\xi) - \hat{z}_{s,d}^\epsilon(\xi)| \leq \frac{M}{\sqrt{d}}|\hat{\sigma}_t(\xi)-\hat{\sigma}_s(\xi)|,\\
    &|\hat{e}_{t,d}(\xi)^\alpha - \hat{e}_{s,d}(\xi)^\alpha| \leq \frac{M}{\sqrt{d}} (|\partial_\alpha\hat{\sigma}_t(\xi)-\partial_\alpha \hat{\sigma}_s(\xi)| + |\hat{\sigma}_t(\xi)-\hat{\sigma}_s(\xi)|).
\end{split} \]
Hence,\ by Lemma~\ref{EOEITnhgoedtihmoIAOIDGHOIDFNhoiAMOIWS},\ there exists a constant $C$,\ independent of $\xi,t,s,d$ and $\alpha$,
\[\begin{split}
    &E|\hat{y}_{t,d}^\epsilon(\xi) - \hat{y}_{s,d}^\epsilon(\xi)| \leq C|t-s|^{1/2},\quad E|\hat{z}_{t,d}^\epsilon(\xi) - \hat{z}_{s,d}^\epsilon(\xi)|\leq C|t-s|^{1/2},\\
    &E|\hat{e}_{t,d}(\xi)^\alpha - \hat{e}_{s,d}(\xi)^\alpha| \leq C|t-s|^{1/2}.
\end{split}\]
From the above, we obtain
\[
    E|\hat{B}_{n,d}^\epsilon(\xi)^{\alpha,\beta} - B_d^\epsilon(\xi)^{\alpha,\beta}| \leq C\Delta^{1/2},\quad E|\hat{D}_{n,d}^\epsilon(\xi)^{\alpha,\beta} - D_d^\epsilon(\xi)^{\alpha,\beta}|\leq C\Delta^{1/2}.
\]
Therefore, $|\hat{B}_{n,d}^\epsilon(\xi)^{\alpha,\beta} - B_d^\epsilon(\xi)^{\alpha,\beta}| \xrightarrow{p} 0$ and $|\hat{D}_{n,d}^\epsilon(\xi)^{\alpha,\beta} - D_d^\epsilon(\xi)^{\alpha,\beta}| \xrightarrow{p} 0$. Consequently,\ the precious arguments yield $\sup_{\xi}|\hat{B}_{n,d}^\epsilon(\xi) -B_d^\epsilon(\xi)|\xrightarrow{p} 0$ and $\sup_\xi|\hat{D}_{n,d}^\epsilon(\xi)-D_d^\epsilon(\xi)| \xrightarrow{p} 0$.

Next, we show that $\sup_{\xi}|B_{d}^\epsilon(\xi) -B^\epsilon(\xi)|\xrightarrow[d \rightarrow \infty]{p} 0$. To prove this, it suffices to show that, for each $\alpha,\beta =1,\ldots,q$,\ $B_{d}^\epsilon(\xi)^{\alpha,\beta}-B^\epsilon(\xi)^{\alpha,\beta}$ converges to $0$ in probability at each $\xi$,\ and that $\sup_{\xi,d}|\nabla_\xi B_{d}^\epsilon(\xi)^{\alpha,\beta}| <\infty$ and $\sup_{\xi} |\nabla_\xi B^\epsilon(\xi)^{\alpha,\beta} | <\infty$.\ However,\ $\sup_{\xi,d}|\nabla_\xi B_{d}^\epsilon(\xi)^{\alpha,\beta}| <\infty$ has already been proved.\ The argument for $\sup_\xi|D_{d}^\epsilon(\xi)-D^\epsilon(\xi)| \xrightarrow[d \rightarrow \infty]{p} 0$ is similar.\ We first note that, as in the previous argument, there exists a random variable $M$ such that for any $\xi,t,\alpha$, and $\gamma$,
$|y_t^\epsilon(\xi)| \leq 4/\epsilon^2$, $|z_t^\epsilon(\xi)| \leq M$, 
$|e_t(\xi)^\alpha| \leq M$,
$|\partial_\gamma y_t^\epsilon(\xi)| \leq M$, $|\partial_\gamma z_t^\epsilon(\xi)| \leq M$, and  $|\partial_\gamma e_t(\xi)^\alpha| \leq M$. This bound implies $\sup_\xi |\nabla_\xi B^\epsilon(\xi)^{\alpha,\beta}|<\infty$ and $\sup_\xi |\nabla_\xi D^\epsilon(\xi)^{\alpha,\beta}|<\infty$ for any $\alpha,\beta$. Furthermore, we obtain
\[\begin{split}
    &|\hat{y}_{t,d}^\epsilon(\xi)-y_t^\epsilon(\xi)| \leq M\left|\int_0^1 \sigma_t^x(\xi)^2 \mu(dx) - \frac{|\hat{\sigma}_t(\xi)|^2}{d}\right|,\\
    &|\hat{z}_{t,d}^\epsilon(\xi)-z_t^\epsilon(\xi)| \leq M\left|\int_0^1 \sigma_t^x(\xi)^2 \mu(dx) - \frac{|\hat{\sigma}_t(\xi)|^2}{d}\right|,\\
    &|\hat{e}_{t,d}^\epsilon(\xi)^\alpha-e_t^\epsilon(\xi)^\alpha| \leq \left|\frac{\partial_\alpha \hat{\sigma}_t(\xi)^\mathsf{T}\hat{\sigma}_t(\xi)}{d}-\int_0^1 \partial_\alpha \sigma_t^x(\xi)\sigma_t^x(\xi)\mu(dx)\right|.
\end{split}\]
As already shown in the proofs of Theorem~\ref{EORigeothimoISNAOFISNrogsdifghaosfigneso} and Lemma~\ref{EORIngotihmoSIFOIGNoaiNDF}, the right-hand side of each inequality converges to $0$ almost surely as $d \to \infty$.
Therefore, by the dominated convergence theorem, we obtain $|B_d^\epsilon(\xi)^{\alpha,\beta}-B^\epsilon(\xi)^{\alpha,\beta}| \xrightarrow[]{p} 0$ and $|D_d^\epsilon(\xi)^{\alpha,\beta}-D^\epsilon(\xi)^{\alpha,\beta}| \xrightarrow[]{p} 0$ for each $\xi$. Consequently, $\sup_\xi|B_d^\epsilon(\xi)-B^\epsilon(\xi)| \xrightarrow[]{p} 0$ and $\sup_\xi|D_d^\epsilon(\xi)-D^\epsilon(\xi)| \xrightarrow[]{p} 0$ holds. Combining these with the previous result completes the proof.
\end{proof}

\FloatBarrier
\section{Simulation and Empirical Analysis}
\label{eoitrgnoIAnoewirgnoISANOSIDFB}
This section has two purposes. 
First, we examine the finite-sample validity of the asymptotic theory by Monte Carlo simulations. 
Second, we conduct an empirical study using SPXW option time-series data. For this empirical analysis, we consider the path-dependent model~\eqref{OWEIrngodfigmoAIBNoisdfbgoiAOINSdf} with the kernel specified as
$k(t)=\eta(t+c)^{H-1/2}$, $\eta>0$, $H\in\mathbb{R}$,
where \(c>0\) is a small regularization constant. We then estimate the parameters of interest \((\eta,H)\).

\FloatBarrier
\subsection{Finite-Sample Validation of the Asymptotic Theory}
We first examine the finite-sample validity of the asymptotic theory by Monte Carlo simulations. 
We consider two parametric kernels,
$k_1(t)\coloneqq\eta\exp(-\xi t)$,
$k_2(t)\coloneqq \eta(t+0.01)^{-\xi}$,
where \(\eta>0\) and \(\xi\in\mathbb{R}\). 
For each kernel, we consider two true parameter values,
$(\eta_0,\xi_0)=(1,-1), (1,1)$.

For each Monte Carlo replication, we first construct the discrete observations
\(\{I_{t_i}^{T_j}\}\) according to the theoretical formula ~\eqref{IWurgnsitugbiUQNIUBAiunsifgubW}--\eqref{OEIRgndosfihgmfdoghinoIWMOSIngodisghoiNSOIN}, with the initial forward variance curve $V_0^\cdot$ is assumed to be constant and equal to $1$. 
We then compute the estimating function \(U_{n,d}^{\epsilon}\) in \eqref{WRgudgtnhnOAIdnfozdifgNSODIndfhg} from these observations with \(\epsilon=10^{-3}\), and obtain the estimator
$(\hat\eta,\hat\xi)$
by minimizing \(U_{n,d}^{\epsilon}(\cdot)\) over \((0,\infty)\times\mathbb{R}\). 
Using the estimated parameter and the same observations \(\{I_{t_i}^{T_j}\}\), we also compute the conditional asymptotic covariance estimator 
\(\hat{\Gamma}_{n,d}^{\epsilon}\) in \eqref{OEirgndsoihoiNOISNdofgisadnfgoi}. 

The asymptotic normality result~\eqref{EWOrigndoghimOAISNFOSIFg} implies that, for sufficiently large \(n\) and \(d\), the studentized statistics
\[\hat Z_{\eta}
\coloneqq \frac{\sqrt{n}(\hat{\eta}-\eta_0)}{\sqrt{(\hat{\Gamma}_{n,d}^{\epsilon})_{11}}}
,\quad
\hat Z_{\xi}
\coloneqq \frac{\sqrt{n}(\hat{\xi}-\xi_0)}{\sqrt{(\hat{\Gamma}_{n,d}^{\epsilon})_{22}}}
\]
should both be approximately standard normal. 
We repeat this procedure $1{,}000$ times. 
The grid is given by
$t_i=i/n$, $i=0,\ldots,n$, $n=10{,}000$,
and
$T_j=j/d$, $j=0,\ldots,d$, $d\asymp n^{0.95}$.
\begin{figure}[!tbp]
\centering
\includegraphics[width=\textwidth]{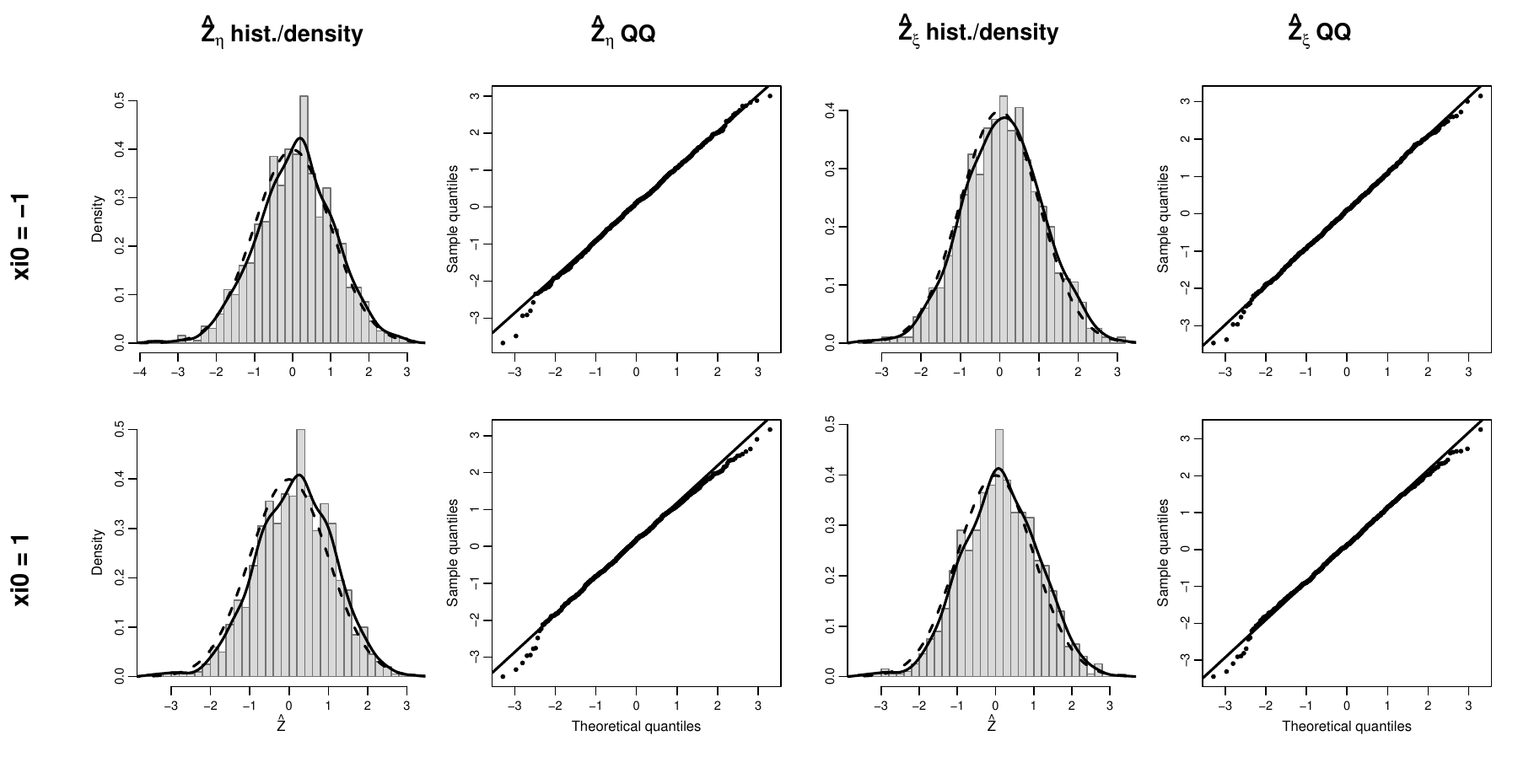}
\caption{
Studentized Monte Carlo diagnostics for \(k_1\) over 1,000 runs: $n=10{,}000$, $d\asymp n^{0.95}$, and $\epsilon = 10^{-3}$.
}
\label{fig:mc_k1}
\end{figure}

\begin{figure}[!tbp]
\centering
\includegraphics[width=\textwidth]{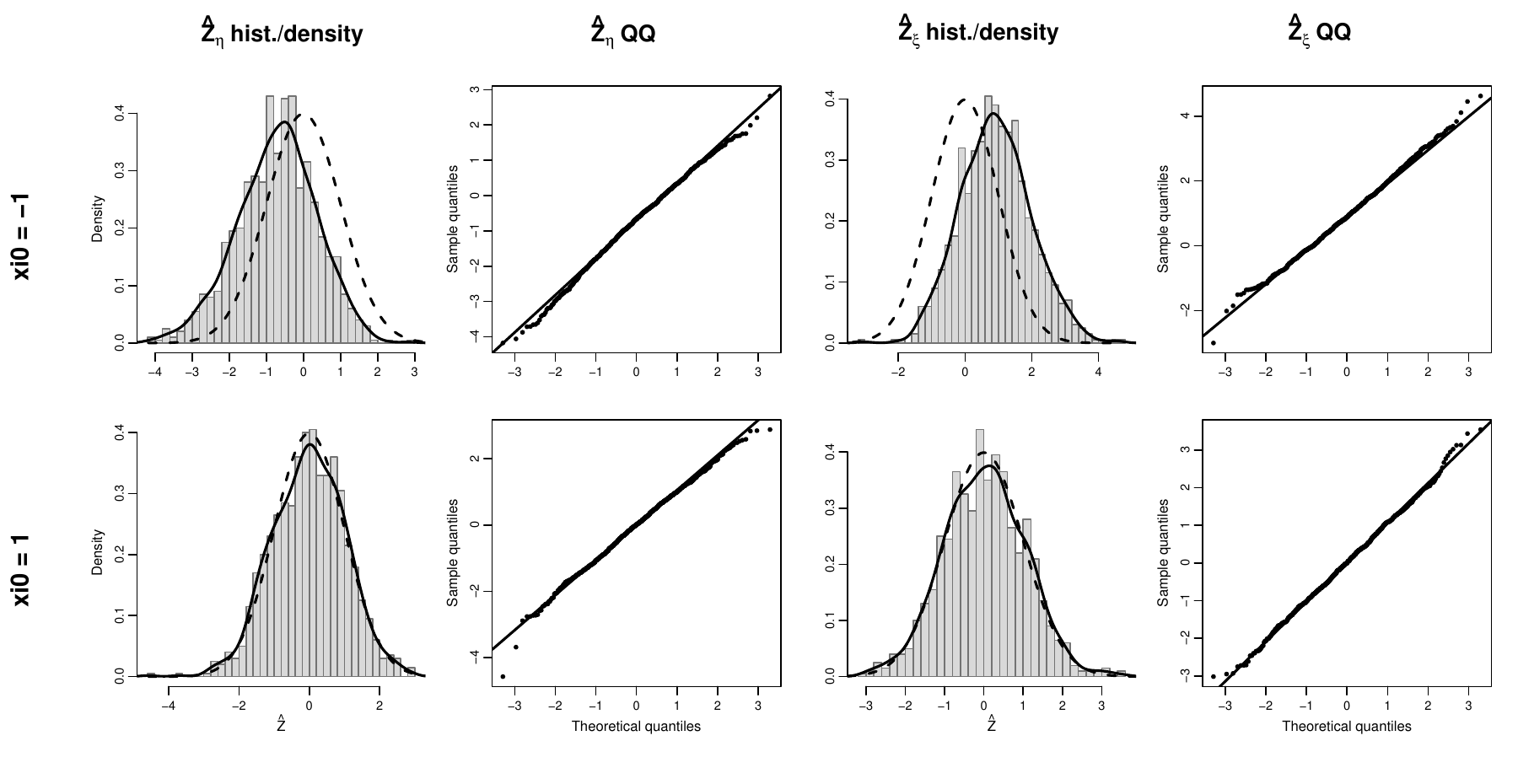}
\caption{
Studentized Monte Carlo diagnostics for \(k_2\) over 1,000 runs: $n=10{,}000$, $d\asymp n^{0.95}$, and $\epsilon = 10^{-3}$.
}
\label{fig:mc_k2}
\end{figure}
Figures~\ref{fig:mc_k1} and~\ref{fig:mc_k2} report the simulation results for the exponential and shifted power-law kernels, respectively. 
In each figure, the left panel reports the histogram and kernel density estimate of the corresponding studentized statistic, together with the standard normal density shown as a dashed line. 
The right panel reports the QQ plot against the standard normal distribution.

For the exponential kernel \(k_1\), the studentized statistics are close to the standard normal benchmark. 
The densities are centered around zero and the QQ plots are almost linear, especially in the central part of the distribution. 
Only mild deviations are observed in the tails. 
These results provide finite-sample support for the asymptotic normality result.

For the shifted power-law kernel \(k_2\), the approximation is somewhat less accurate than for the exponential kernel. 
In particular, when \(\xi_0=-1\), the QQ plots remain approximately linear, but the density plots indicate a visible shift relative to the standard normal benchmark. 
This suggests that the studentized statistics are close to normal in shape, but exhibit a finite-sample location bias.  
When \(\xi_0=1\), the approximation is considerably better, with the QQ plots showing a good fit in the central region. 

Overall, the simulations are broadly consistent with the asymptotic normality result. 
At the same time, they show that the quality of the finite-sample approximation depends on the kernel and the parameter region, with the shifted power-law kernel under \(\xi_0=-1\) exhibiting the most visible centering error.

\FloatBarrier
\subsection{Empirical Estimation of the Path-Dependent Model}
We now apply the proposed method to SPXW option time-series data. 
The sample covers options traded from January to the end of December 2021, with maturities restricted to the same period, so that both the trading-date and maturity dimensions span one year.

Following \cite{fukasawa-murayama}, we compute discrete observations of the cumulative forward variance 
\(\{I_{t_i}^{T_j}\}\) from the option data, where missing values are imputed by PCHIP interpolation across maturities for each fixed \(t_i\). We then use these observations as inputs for the proposed estimation procedure. We take the path-dependent model~\eqref{OWEIrngodfigmoAIBNoisdfbgoiAOINSdf} as the estimation target, with the kernel specified as
$k(t)=\eta(t+0.01)^{H-1/2}$, $\eta>0$, $H\in\mathbb{R}$.
We jointly estimate the kernel parameters \((\eta,H)\).
\small
\begin{table}[!tbp]
\centering
\caption{Empirical estimates of the path-dependent model for different values of \(\epsilon\).}
\label{tab:empirical_joint_estimation}
\begin{tabular}{cccc}
\toprule
\(\epsilon\) & \(\hat{\eta}\) & \(\hat{H}\) & 95\% CI for \(H\) \\
\midrule
\(10^{-2}\) & 0.102 & -0.859  & [-1.946, 0.228] \\
\(10^{-3}\) & 0.031 & -1.259  & [-1.906, -0.611] \\
\(10^{-4}\) & 0.104 & -0.789 & [-1.147, -0.431] \\
\(10^{-5}\) & 0.265 & -0.375 & [-0.607, -0.142] \\
\bottomrule
\end{tabular}
\end{table}
Table~\ref{tab:empirical_joint_estimation} reports the estimation results for 
\(\epsilon\in\{10^{-2},10^{-3},10^{-4},10^{-5}\}\).
For all values of \(\epsilon\), the estimate of \(H\) is negative. 
This is in line with the empirical findings of \cite{AbiJaberLi2025}, who report a negative calibrated value of \(H\) for a path-dependent specification.
At the same time, the estimates show some sensitivity to the choice of \(\epsilon\). 
In particular, for \(\epsilon=10^{-3}\), the estimate of the scale parameter is 
\(\hat{\eta}=0.031\), which is substantially smaller than the estimates obtained for the other values of \(\epsilon\). 
The confidence intervals for \(H\) are also relatively wide in some cases. 

\FloatBarrier
\subsection{Instability of the Joint Estimation}
The empirical results in the previous subsection suggest that the joint estimation of 
\((\eta,H)\) is sensitive to the choice of \(\epsilon\). 
These observations indicate that the simultaneous estimation of the scale and shape parameters may be unstable in the present empirical design.

To examine whether this instability is specific to the observed sample or reflects a more systematic finite-sample feature, we conduct a Monte Carlo experiment calibrated to the empirical setting. 
To construct the simulated samples, we start from the observed discrete cumulative forward variance 
\(\{I_{t_i}^{T_j}\}\), computed from the option data. 
We recover the corresponding initial forward variance curve \(V_0^{\cdot}\) and then reconstruct 
\(\{I_{t_i}^{T_j}\}\) using the theoretical formula~\eqref{IWurgnsitugbiUQNIUBAiunsifgubW}--\eqref{OEIRgndosfihgmfdoghinoIWMOSIngodisghoiNSOIN}. 
This construction preserves the main time-to-maturity structure of the empirical data. 
The true scale parameter is fixed at \(\eta_0=0.1\), which is of the same order as the empirical scale estimates and serves as a representative normalization for this diagnostic experiment. 
We consider several true values of \(H\), namely \(H_0\in\{-1.2,-0.8,-0.4\}\), and jointly estimate \((\eta,H)\) for 
\(\epsilon\in\{10^{-2},10^{-3},10^{-4},10^{-5}\}\). 
Tables~\ref{tab:eps_H_minus_1_200}--\ref{tab:eps_H_minus_0_400} report the results.
In these tables, \(\hat{Z}\) denotes the studentized statistic for the estimator of \(H\), and \(\hat{H}_{\mathrm{low}}\) and \(\hat{H}_{\mathrm{up}}\) denote the lower and upper endpoints of the two-sided 95\% confidence interval for \(H\), respectively. 
\small
\begin{table}[!tbp]
\caption{\label{tab:eps_H_minus_1_200}Monte Carlo performance of the joint estimator under \(\eta_0=0.1\) and \(H_0=-1.2\).}
\centering
\begin{tabular}{ccrrrrrr}
\toprule
$\epsilon$ & $E[\hat{\eta}]$ & $\hat{H}_{Bias}$ & $\hat{H}_{RMSE}$ & $E[\hat{Z}]$ & $\mathrm{Var}(\hat{Z})$ & $P(\hat{H}_{\mathrm{low}} \le H \le \hat{H}_{\mathrm{up}})$ \\
\midrule
$10^{-2}$ & 47.74 & 4.86 & 19.21 & -1.24 & 1504.6 & 0.385 \\
$10^{-3}$ & 10.00 & 0.92 & 3.10 & -2.05 & 1465.8 & 0.220 \\
$10^{-4}$ & 0.53 & 0.48 & 0.91 & -0.76 & 2189.2 & 0.101 \\
$10^{-5}$ & 0.36 & 0.69 & 0.79 & 5.09 & 15.0 & 0.022 \\
\bottomrule
\end{tabular}
\end{table}
\small
\begin{table}[!tbp]

\caption{\label{tab:eps_H_minus_0_800}Monte Carlo performance of the joint estimator under \(\eta_0=0.1\) and \(H_0=-0.8\).}
\centering
\begin{tabular}{ccrrrrrr}
\toprule
$\epsilon$ & $E[\hat{\eta}]$ & $\hat{H}_{Bias}$ & $\hat{H}_{RMSE}$ & $E[\hat{Z}]$ & $\mathrm{Var}(\hat{Z})$ & $P(\hat{H}_{\mathrm{low}} \le H \le \hat{H}_{\mathrm{up}})$\\
\midrule
$10^{-2}$ & 24.21 & 5.91 & 29.26 & -0.23 & 171.4 & 0.697 \\
$10^{-3}$ & 5.80 & 1.51 & 5.49 & 0.52 & 17.8 & 0.473 \\
$10^{-4}$ & 0.42 & 0.53 & 0.90 & 2.29 & 10.3 & 0.207 \\
$10^{-5}$ & 0.57 & 0.99 & 1.05 & 3.44 & 2.1 & 0.093 \\
\bottomrule
\end{tabular}
\end{table}
\small
\begin{table}[!tbp]
\caption{\label{tab:eps_H_minus_0_400}Monte Carlo performance of the joint estimator under \(\eta_0=0.1\) and \(H_0=-0.4\).}
\centering
\begin{tabular}{ccrrrrrr}
\toprule
$\epsilon$ & $E[\hat{\eta}]$ & $\hat{H}_{Bias}$ & $\hat{H}_{RMSE}$ & $E[\hat{Z}]$ & $\mathrm{Var}(\hat{Z})$ & $P(\hat{H}_{\mathrm{low}} \le H \le \hat{H}_{\mathrm{up}})$ \\
\midrule
$10^{-2}$ & 3.16 & 2.67 & 34.51 & -2.078 & 205.1 & 0.828 \\
$10^{-3}$ & 0.32 & 0.51 & 0.83 & -0.49 & 37.5 & 0.825 \\
$10^{-4}$ & 0.37 & 0.75 & 0.83 & 1.07 & 1.4 & 0.832 \\
$10^{-5}$ & 0.63 & 1.06 & 1.06 & 1.64 & 0.0 & 1.000 \\
\bottomrule
\end{tabular}
\end{table}
The simulation results show substantial instability of the joint estimator in this empirically calibrated design. 
The estimate of the scale parameter varies considerably across values of \(\epsilon\), even though the true value is fixed at \(\eta_0=0.1\). 
The bias and RMSE of \(\hat H\) are also large in several cases. 
The studentized statistics deviate substantially from the standard normal benchmark, as reflected in their means and variances. 
Moreover, the empirical coverage probabilities of the confidence intervals for \(H\) are far from the nominal level.

Overall, these results do not formally establish weak identification. 
However, they point to a practical difficulty in separating the scale and shape directions in the present empirical design. 

\FloatBarrier
\subsection{Fixed-\texorpdfstring{$\eta$}{eta} Estimation}
 As a complementary analysis, we therefore consider the estimation of \(H\) with the scale parameter fixed. 
Throughout this subsection, we fix \(\eta=0.1\), which is of the same order as the empirical scale estimates obtained above.

A remaining practical issue is the choice of the tuning parameter \(\epsilon\). 
To select a reasonable value, we use the same empirically calibrated Monte Carlo design as in the previous subsection, but estimate only \(H\) while keeping \(\eta=0.1\) fixed. 
We consider the same true values \(H_0\in\{-1.2,-0.8,-0.4\}\) and compare 
\(\epsilon\in\{10^{-2},10^{-3},10^{-4},10^{-5}\}\). 
Tables~\ref{tab:fixed_eta_H_minus_1200}--\ref{tab:fixed_eta_H_minus_0400} report the results.
The fixed-\(\eta\) estimator is substantially more stable than the joint estimator. 
Across the three values of \(H_0\), the bias and RMSE of \(\hat H\) are much smaller than those obtained in the joint estimation experiment. 
The studentized statistics are also closer to the standard normal benchmark, and the empirical coverage probabilities are more reasonable for moderate choices of \(\epsilon\). 
Among the values considered, \(\epsilon=10^{-3}\) provides the most balanced performance across the three designs: it yields small bias and RMSE, while avoiding the severe deterioration observed for very small values of \(\epsilon\), such as \(10^{-5}\).
\small
\begin{table}[!tbp]

\caption{\label{tab:fixed_eta_H_minus_1200}Monte Carlo performance of the fixed-\(\eta \) estimator under \(H_0=-1.2\).}
\centering
\begin{tabular}{ccrrrrrr}
\toprule
$\epsilon$  & $\hat{H}_{Bias}$ & $\hat{H}_{RMSE}$ & $E[\hat{Z}]$ & $\mathrm{Var}(\hat{Z})$ & $P(\hat{H}_{\mathrm{low}} \le H \le \hat{H}_{\mathrm{up}})$ \\
\midrule
$10^{-2}$ & 0.08 & 0.15 & 0.54 & 1.79 & 0.910 \\
$10^{-3}$ & -0.01 & 0.08 & -0.23 & 2.81 & 0.858 \\
$10^{-4}$ & -0.04 & 0.08 & -0.91 & 2.94 & 0.808 \\
$10^{-5}$ & -0.11 & 0.12 & -2.76 & 2.14 & 0.278 \\
\bottomrule
\end{tabular}
\end{table}
\small
\begin{table}[!tbp]

\caption{\label{tab:fixed_eta_H_minus_8}Monte Carlo performance of the fixed-\(\eta\) estimator under \(H_0=-0.8\).}
\centering
\begin{tabular}{ccrrrrrr}
\toprule
$\epsilon$  & $\hat{H}_{Bias}$ & $\hat{H}_{RMSE}$ & $E[\hat{Z}]$ & $\mathrm{Var}(\hat{Z})$ & $P(\hat{H}_{\mathrm{low}} \le H \le \hat{H}_{\mathrm{up}})$ \\
\midrule
$10^{-2}$  & 0.02 & 0.08 & 0.26 & 0.67 & 0.973 \\
$10^{-3}$  & -0.01 & 0.06 & -0.12 & 0.90 & 0.952 \\
$10^{-4}$  & -0.04 & 0.06 & -0.95 & 1.10 & 0.845 \\
$10^{-5}$  & -0.16 & 0.17 & -4.18 & 1.64 & 0.028 \\
\bottomrule
\end{tabular}
\end{table}
\small
\begin{table}[!tbp]

\caption{\label{tab:fixed_eta_H_minus_0400}Monte Carlo performance of the fixed-\(\eta\) estimator under \(H_0=-0.4\).}
\centering
\begin{tabular}{ccrrrrrr}
\toprule
$\epsilon$ & $\hat{H}_{Bias}$ & $\hat{H}_{RMSE}$ & $E[\hat{Z}]$ & $\mathrm{Var}(\hat{Z})$ & $P(\hat{H}_{\mathrm{low}} \le H \le \hat{H}_{\mathrm{up}})$ \\
\midrule
$10^{-2}$  & -0.09 & 0.10 & -1.12 & 0.50 & 0.926 \\
$10^{-3}$  & -0.05 & 0.08 & -0.75 & 0.71 & 0.925 \\
$10^{-4}$  & -0.09 & 0.10 & -1.81 & 0.64 & 0.567\\
$10^{-5}$ & -0.30 & 0.30 & -7.21 & 1.33 & 0.000 \\
\bottomrule
\end{tabular}
\end{table}

Based on this comparison, we use \(\epsilon=10^{-3}\) for the fixed-\(\eta\) empirical estimation. 
As an additional diagnostic check, 
\begin{figure}[!tbp]

\centering

\includegraphics[width=0.9\textwidth]{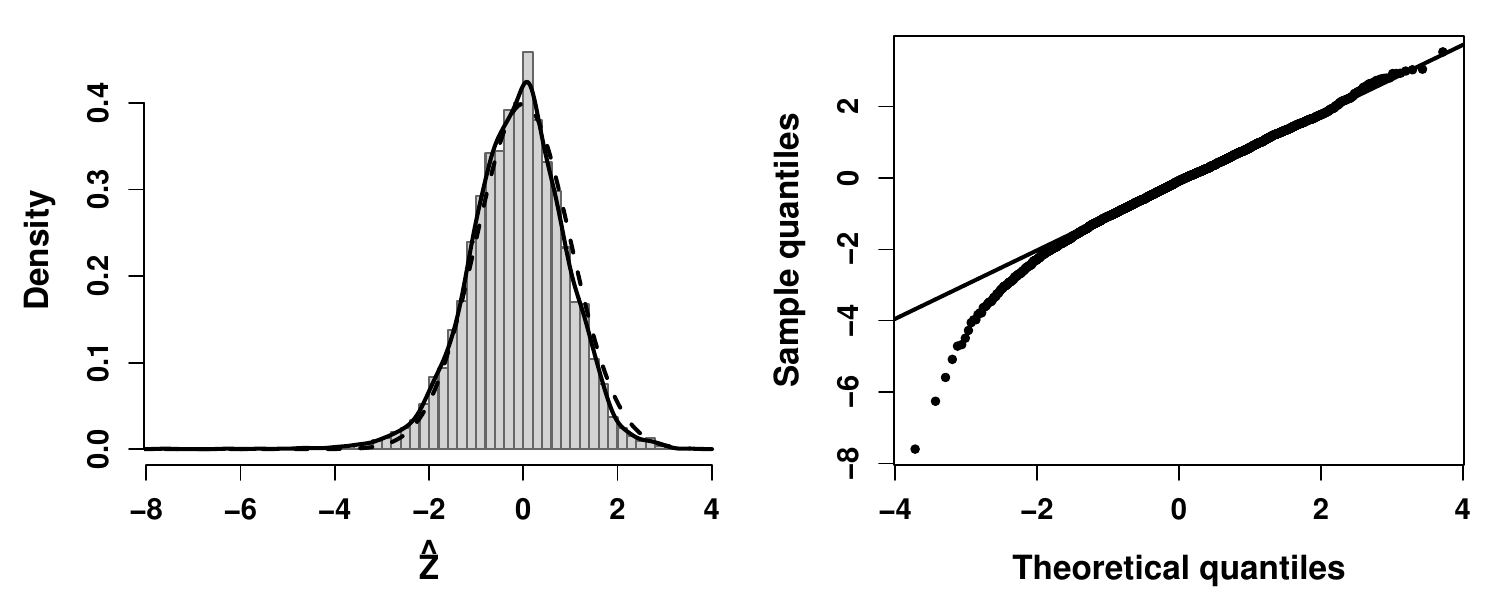}

\caption{Density and QQ plots of the studentized fixed-$\eta$ estimator under \(H_0=-0.8\) and \(\epsilon=10^{-3}\), based on 5,000 empirically calibrated Monte Carlo replications.}
\label{fig:fixed_eta_studentized}
\end{figure}
Figure~\ref{fig:fixed_eta_studentized} reports the density and QQ plots of the studentized fixed-\(\eta\) estimator of \(H\), based on 5000 Monte Carlo replications with \(\eta=0.1\), \(H_0=-0.8\), and \(\epsilon=10^{-3}\). 
The normal approximation is reasonably accurate around the center of the distribution, although the QQ plot shows a heavier left tail. 
Overall, the fixed-\(\eta\) inference is substantially more stable than the joint inference, although the normal approximation is not perfect in the tails. 

Finally, we estimate \(H\) from the SPXW data with \(\eta\) fixed at 0.1 and \(\epsilon=10^{-3}\). 
The resulting estimate is
$\hat H = -0.855$,
with the two-sided 95\% confidence interval
$[-0.906, -0.803]$.
The estimate is negative and is in line with the sign obtained from the joint estimation results, while the confidence interval is substantially narrower. 
This suggests that fixing the scale parameter provides a more stable inference on the shape parameter \(H\) in the present empirical design.

Overall, the simulation and empirical results lead to two main conclusions. 
First, the proposed studentization is broadly consistent with the asymptotic normality theory, although finite-sample centering errors can appear for some kernel specifications. 
Second, in the empirical SPXW application, the joint estimation of the scale and shape parameters appears practically unstable, while the fixed-scale analysis provides more stable inference on \(H\). 
Across the specifications considered, the empirical estimates point to a negative value of \(H\), which is in line with previous empirical evidence on path-dependent specifications.

\bibliographystyle{plain}
\bibliography{mybib}

\end{document}